\documentclass[final]{siamltex}

\usepackage{amssymb,latexsym,amsmath}
\usepackage{epsfig}
\usepackage{caption}
 
 \usepackage{amsmath}
\usepackage{amsfonts}
\usepackage{array}
\usepackage{dsfont}
\usepackage{amssymb}

\usepackage{algorithm}
\usepackage{algorithmic}

\newtheorem{assumption}{Assumption}[section]

 \newtheorem{remark}{Remark}[section]
\usepackage[colorlinks=true,breaklinks=true,bookmarks=true,urlcolor=blue,
     citecolor=blue,linkcolor=blue,bookmarksopen=false,draft=false]{hyperref}

\def\Pr{\mathop{\rm Pr}}

\usepackage[utf8]{inputenc}

\usepackage{titlesec}
\titleformat{\chapter}[display]{\Large\bfseries\centering}%
    {\chaptername~\thechapter}{1ex}{}[\titlerule]

\begin{document}

\sloppy

\title{Quantizer Design for Finite Model Approximations, Model Learning, and Quantized Q-Learning for MDPs with Unbounded Spaces \thanks{
This research was supported in part by
the Natural Sciences and Engineering Research Council (NSERC) of Canada.}
}

\author{Osman Bi\c{c}er, Ali D. Kara and Serdar Y\"uksel
\thanks{A.D. Kara is with the Department of Mathematics, Florida State University. Osman Bi\c{c}er and S. Y\"uksel are with the Department of Mathematics and Statistics,
     Queen's University, Kingston, ON, Canada}
     }

\maketitle

\begin{abstract} In this paper, for Markov decision processes (MDPs) with unbounded state spaces we present refined upper bounds presented in [Kara et. al. JMLR'23] on finite model approximation errors via optimizing the quantizers used for finite model approximations. We also consider implications on quantizer design for quantized Q-learning and empirical model learning, and the performance of policies obtained via Q-learning where the quantized state is treated as the state itself. We highlight the distinctions between planning, where approximating MDPs can be independently designed, and learning (either via Q-learning or empirical model learning), where approximating MDPs are restricted to be defined by invariant measures of Markov chains under exploration policies, leading to significant subtleties on quantizer design performance, even though asymptotic near optimality can be established under both setups. In particular, under Lyapunov growth conditions, we obtain explicit upper bounds which decay to zero as the number of bins approaches infinity. \end{abstract}



\section{Introduction}
It has been recently shown that one can obtain finite approximations via state and action quantization for Markov decision processes (MDPs) with uncountable Polish spaces, as well as run both empirical model learning and Q-learning to arrive at near optimal solutions (see.e.g. \cite{SaLiYuSpringer,SaYuLi15c}). Such studies have primarily focused on the case with compact spaces and uniform quantization and with only asymptotic convergence results for the case with non-compact spaces. To this end, the goal of this paper is to design quantizers in such a finite model approximation and learning framework for continuous space MDPs. %Notably, this paper develops a mathematical framework and formulation paving the way for non-uniform quantizer design and associated performance bounds, which are then also applicable to MDPs with non-compact spaces. 

\subsection{Related Literature}
For stochastic control problems with continuous state and action spaces, approximations are inevitable for computational methods.  A common approach in reinforcement learning is function approximation, where the value function (Q-function) of the control problem is approximated using a parametrized family of functions. Convergence of {\it policy evaluation} methods are known under linear function approximation where the parametrized family of functions are formed by the span of finitely many basis functions \cite{tsitsiklis1997analysis}. However, learning  optimal Q-functions with linear function approximation is known to be unstable in general \cite{baird1995residual} except in special cases. For general linear function approximation, \cite{meyn2024projected} has shown that under a certain class of exploration policies, optimal Q-value learning  with linear function approximation remains bounded. The special cases where the convergence can be guaranteed include when (i) the exploration policy is already close to the greedy policy of the learning iterations \cite{melo2008analysis}, and (ii) the stage-wise cost function, the transition model, and thus the optimal Q-functions are perfectly represented by the basis functions, i.e. they belong to the span of the basis functions \cite{jin2023provably, Ruszczy2024}. These assumptions, however, are restrictive in general, since it is unrealistic to assume near-optimal exploration or perfect linear representability.  

A particularly powerful, though computationally demanding, special case is when the basis functions are indicator functions of  quantization bins, which form an orthonormal basis, and the convergence analysis can mathematically be justified. In particular, \cite{kara2023qlearning} showed that the Q-values learned under quantization based learning correspond to an auxiliary finite control problem, which is a finite approximation of the original control problem with a particular weight measure on the quantization bins. This observation has an important consequence that the model-based approaches using state space quantization and the quantized Q-learning coincide and can be used interchangeably, even under mild continuity conditions on transition kernels \cite{SaYuLi15c}. 

For quantization based learning methods while often convergence is studied,  error analysis regarding the limit of the stochastic iterates is typically not studied in general or an error analysis is not provided at all. Early studies in this direction include \cite{singh1995reinforcement,gaskett1999q}. \cite{CsabaSmart}  generalized these by the use of Q function interpolators that are sufficiently regular (defined by non-expansiveness) in their parametric representation and established both convergence and optimality properties. Error analysis of the learned value functions with respect to the true value functions under quantization methods is studied in \cite{NNQlearning}, where the authors established finite-sample guarantees for quantized Q-learning with nearest-neighbor mappings, assuming transition models admit continuous densities with respect to the Lebesgue measure.  

In \cite{SaYuLi15c,kara2023qlearning,kara2024qlearning}, it was shown that the approximate value function, as well as the performance of the resulting policies, are nearly optimal under weakly continuous (or weak Feller) models, arguably the weakest assumption ensuring existence of optimal solutions and consistency of aproximations.  

Another direction of research with provable loss guarantees is kernel-based methods \cite{ormoneit2002kernel,ormoneit2002average,barreto2016practical,domingues2021kernel,zhou2024robustness,DuPr14} where the cost and the transition models are estimated via empirical kernel regression or otherwise simulations are used to update the Bellman optimality equations.  One can then find control policies based on the learned models using model-based approaches. 

Most prior work focuses on approximation of control problems with compact spaces. There is a limited number of provable guarantees for the approximation of control problems with unbounded state spaces. In \cite{SaYuLi15c}  {\it asymptotic} near optimality of the quantization based approximations was established for non-compact  state spaces as the quantization rate grows to infinity. However, there is no convergence rate guarantees for this approach in terms of  expected quantization error on the state space and learning was not studied in this context. 

Moreover, previous work generally assumes a fixed quantization scheme, without focusing on the design of the quantizers. In this design direction, \cite{sinclair2023adaptive} studied adaptive quantization methods for compact spaces under Wasserstein continuous transitions, proposing progressively refined partitions based on `relevance' of the partition sets. The resulting quantization scheme is  not optimal in general, as refinement is localized rather than globally optimized.

\subsection{Model and Cost Criteria}

Let $\mathds{X} \subset \mathds{R}^n$ be a Borel set in which the elements of a controlled Markov chain $\{X_t,\, t \in \mathbb{Z}_+\}$ take values for some $n<\infty$.  Here and throughout the paper, $\mathbb{Z}_+$ denotes the set of non-negative integers and $\mathds{N}$ denotes the set of positive integers. Let $\mathds{U}$, the action space, be a compact Borel subset of some Euclidean space, from which the sequence of control action variables $\{U_t,\, t \in \mathbb{Z}_+\}$ take values. 

The $\{U_t, \, t \in \mathbb{Z}_+\}$, are generated via admissible control policies: An {\em admissible policy} $\gamma$ is a sequence of control functions $\{\gamma_t,\, t\in \mathbb{Z}_+\}$ such that $\gamma_t$ is measurable on the $\sigma$-algebra generated by the information variables
\[
I_t=\{X_0,\ldots,X_t,U_0,\ldots,U_{t-1}\}, \quad t \in \mathds{N}, \quad
  \quad I_0=\{X_0\},
\]
where
\begin{equation}
\label{eq_control}
U_t=\gamma_t(I_t),\quad t\in \mathbb{Z}_+,
\end{equation}
are the $\mathds{U}$-valued control
actions.
%\[X_{[0,t]} = \{X_s,\, 0 \leq s \leq t \}, \quad U_{[0,t-1]} =
  %\{U_s, \, 0 \leq s \leq t-1 \}.\]
\noindent We define $\Gamma$ to be the set of all such admissible policies.

The joint distribution of the state and control
processes is then completely determined by (\ref{eq_control}), the initial probability measure of $X_0$, and the following
relationship:
\begin{align}\label{eq_evol}
& \Pr\biggl( X_t\in B \, \bigg|\, (X,U)_{[0,t-1]}=(x,u)_{[0,t-1]} \biggr) = \int_B \mathcal{T}( dx_t|x_{t-1}, u_{t-1}),  B\in \mathcal{B}(\mathds{X}), t\in \mathds{N},
\end{align}
where $\mathcal{T}(\cdot|x,u)$ is a stochastic kernel  (that is, a regular conditional probability measure) from $\mathds{X}\times \mathds{U}$ to $\mathds{X}$, $\mathcal{B}(\mathds{X})$ is the Borel $\sigma$-algebra of $\mathds{X}$, and $(X,U)_{[0,t-1]}$ is the set of state-action pairs up until $t-1$. We will be interested in the following performance criteria: The first one is the infinite-horizon discounted expected cost
  \begin{align}
    J_{\beta}(x_0,\gamma)= E_{x_0}^{{\mathcal{T}},\gamma}\left[\sum_{t=0}^{\infty} \beta^tc(X_t,U_t)\right]
  \end{align}
  where $0<\beta<1$ is the discount factor, $c:\mathds{X}\times\mathds{U}\to\mathds{R}$ is the stage-wise continuous and bounded cost function, and $E_{x_0}^{{\mathcal{T}},\gamma}$ denotes the expectation with initial state $x_0$ and transition kernel $\mathcal{T}$ under policy $\gamma$. Furthermore, for any initial state $X_0=x_0$, the optimal value function is defined by
\begin{align*}
  J_{\beta}^*(x_0)&=\inf_{\gamma\in\Gamma} J_{\beta}(x_0,\gamma).
\end{align*}
The second objective is the infinite horizon average cost criterion
\begin{eqnarray}\label{AverageCostProblemDef}
J_{avg}^*(x):= \inf_{\gamma} J_{avg}(x,\gamma) = \inf_{\gamma \in \Gamma} \limsup_{T \to \infty} {1 \over T} E^{\gamma}_x [\sum_{t=0}^{T-1} c(x_t,u_t)].
\end{eqnarray} 

%To calculate the optimal value function and the optimal control policy, various numerical approaches can be adopted, e.g., value iteration, policy iteration, linear programming under the assumption that the transition probability $\mathcal{T}$ and the cost function $c$ are known. If the model is unknown, a powerful and popular tool is the Q-learning algorithm by \cite{Watkins}. The Q-learning algorithm provides an iterative approach that is guaranteed to converge under mild assumptions if the model is finite and if the controller has access to the state and cost realizations.

\subsection{Contributions}

In this paper, we make the following contributions: {\bf (i)} In Theorem \ref{theorem:refined_error_bound}, we refine and computationally improve the upper bounds given in \cite{kara2023qlearning} which involve admissible policies to ones that only involve stationary policies. This facilitates an analysis involving occupation measures as well as invariant measures for discounted and ergodic cost criteria, respectively; which is then utilized later in the paper.  {\bf (ii)} In Theorems \ref{theorem:optimized_error_bound} \ref{theorem:error_bound_average_alt}, we derive bounds in terms of occupation measures, to represent the loss in terms of occupation and invariant measures, respectively for the discounted and average cost criteria. In  Corollary \ref{median_bound_disc}, we optimize the quantization design by choosing the representative points as medians (in the $\ell_1$ sense) with respect to the occupation measures and arrive at an explicit error bound. We thus provide a convergence rate analysis of quantization based approximations which is convenient for stochastic analysis and which is also applicable for {\it non-compact} state spaces. {\bf (iii)} In Theorem \ref{theorem:uniform_quantizer_lyapunov}, under Foster-Lyapunov conditions, we derive explicit error bounds which decay to zero as the quantization gets finer and show better performance than uniform quantization in MDPs with non-compact state spaces under discounted cost criteria. We extend our analysis to the average cost criterion by obtaining explicit error bounds  in Theorem \ref{theorem:uniform_quantizer_lyapunov_average}.  {\bf (iv)} In Theorem \ref{theorem:lyapunov_learning}, we extend our analysis on quantizer design to quantized Q-learning and empirical model learning (which are equivalent in performance). We obtain the error bounds which depend, unlike the planning problem above, on the invariant measure induced by the exploration policy in Theorem \ref{theorem:error_quantized_q_learning}. For non-compact spaces, in Theorem \ref{theorem:lyapunov_learning}, we show that the approximation error diminishes as quantization becomes finer under Foster–Lyapunov conditions, despite constraints on the dependence of the weighting measures on the exploration policy and quantization bins, and we derive explicit error bounds.

\section{Refined Error Bounds on Finite Model Approximations}

\subsection{Approximate Model Construction}\label{finiteModelAppC}

We start with the the approach introduced in \cite{saldi2018finite,saldi2017optimality}, where we partition the state space $\mathbb{X}$ into $M$ disjoint subsets $\{B_i\}_{i=1}^M$, such that $\bigcup_{i=1}^M B_i = \mathbb{X}$ and $B_i \cap B_j = \emptyset$ for $i \neq j$. For each subset $B_i$, we select a representative state $y_i \in B_i$. The finite set $\mathbb{Y} = \{ y_1, y_2, \dots, y_M \}$ serves as the quantized state space. The quantizer mapping $q: \mathbb{X} \to \mathbb{Y}$ is defined by
\begin{equation*}
q(x) = y_i \quad \text{if } x \in B_i.
\end{equation*}

We introduce a probability measure $\pi \in \mathcal{P}(\mathbb{X})$ over $\mathbb{X}$, ensuring that $\pi(B_i) > 0$ for each $B_i$. This measure allows us to define normalized measures for each quantization bin $B_i$:
\begin{equation*}
\hat{\pi}_{y_i}(A) = \frac{\pi(A)}{\pi(B_i)}, \quad \forall A \subseteq B_i, \quad \forall i \in \{1, \ldots, M\}.
\end{equation*}
Using these normalized measures, we define the stage-wise cost and transition kernel for the finite-state MDP:
\begin{equation*}
C^*(y_i, u) = \int_{B_i} c(x, u) \hat{\pi}_{y_i}(dx),
\end{equation*}
\begin{equation*}
P^*(y_j | y_i, u) = \int_{B_i} \mathcal{T}(B_j | x, u) \hat{\pi}_{y_i}(dx).
\end{equation*}

The finite-state value function $\hat{J}_{\beta}: \mathds{Y} \rightarrow \mathbb{R}$ satisfies the dynamic programming equation:
\begin{equation*}
\hat{J}_{\beta}(y) = \inf_{u \in \mathbb{U}} \left\{ C^*(y, u) + \beta \sum_{z \in \mathbb{Y}} \hat{J}_{\beta}(z) P^*(z | y, u) \right\}.
\end{equation*}
We extend $\hat{J}_{\beta}$ to $\mathbb{X}$ by setting $\hat{J}_{\beta}(x) = \hat{J}_{\beta}(q(x))$ for all $x \in \mathbb{X}$.

Under certain regularity conditions, we will see that the quantization error can be efficiently bounded by the loss function $L: \mathbb{X} \rightarrow \mathbb{R}$:
\begin{equation}\label{eq:loss_function}
L(x) = \int_{B_i} \|x - x'\|_1 \hat{\pi}_{y_i}(dx'), \quad \forall x \in B_i
\end{equation}
where $\|\cdot\|_1$ denotes the $\ell_1$ norm in $\mathds{R}^n$. 

\begin{remark}
We note that the loss function (\ref{eq:loss_function}) above is also often called the potential function of the measure $\hat{\pi}_{y_i}$ evaluated at $x$, in the mathematical statistics and probability theory literatures, see e.g. \cite{beiglbock2022approximation}.
\end{remark}

\subsection{Refined Error Bounds for the Discounted Cost Criterion}
\label{section:refined_discounted}
To derive error bounds, we make the following assumptions:

\begin{assumption} \label{assumption:model}
    The MDP $(\mathds{X},\mathds{U},\mathcal{T},c)$ satisfies the following:
    \begin{itemize}
        \item[(i)] The action space $\mathds{U}$ is compact. 
        \item[(ii)] The stage-wise cost function $c$ is nonnegative, bounded and continuous on both $\mathds{X}$ and $\mathds{U}$.
\item[(iii)] The kernel ${\cal T}$ is weak Feller, that is, for every $g \in C_b(\mathds{X})$ (continuous and bounded function), the map
\[\mathds{X} \times \mathds{U} \ni (x,u) \mapsto \int g(x_1){\cal T}(dx_1 | x_0=x,u_0=u) \in \mathbb{R} \]
is continuous.
    \end{itemize}
\end{assumption}

The above ensures that optimal policies exist. Furthermore, by \cite[Lemma 3.19]{SaLiYuSpringer} and \cite[Theorem 3.16]{SaLiYuSpringer} (see also \cite{saldi2014near}), any MDP with a weakly continuous transition probability kernel can be approximated by an MDP with finite action spaces. Accordingly, in the sequel, we assume that $\mathds{U}$ is finite.

%\begin{remark}
%The boundedness assumption  on the stage-wise cost function $c(\cdot,\cdot)$ can be relaxed. In particular, the results in this paper only require  boundedness of the value function of the approximate model, i.e. $\|\hat{J}_\beta\|_\infty < \infty$. This boundedness can be ensured if the approximate stage-wise cost function $C^*(y_i,u)$ is bounded for each $y_i$ and $u$. %Since the true cost $c(x,u)$ is continuous, it remains bounded over compact sets.  For the overflow bins, i.e. those which are not bounded, one can choose the weighting measure $\hat{\pi}(\cdot)$ to assign full measure to a bounded set, which is then sufficient to guarantee the boundedness of the approximate value function $\hat{J}_\beta$. %However, as we will see later in the paper, for models learned via Q-learning, we do not have control over the choice of the weighting measure. 
%%To streamline the analysis we assume that the original stage-wise cost $c$ is bounded, although it is possible to relax this assumption. 
%\end{remark}

\begin{assumption}\label{assumption:lipschitz_cost_and_transitions} 
The transition kernel and the stage-wise cost function satisfies the following:
\begin{itemize}
    \item[(i)] $c(x, u)$ is Lipschitz continuous in $x$. There exists a constant $\alpha_c > 0$ such that
\[
|c(x, u) - c(x', u)| \leq \alpha_c \|x - x'\|_1, \quad \forall x, x' \in \mathds{X}, \forall u \in \mathds{U}.
\]
\item[(ii)] $\mathcal{T}(\cdot  | x, u)$ is Lipschitz continuous in $x$ under the total variation distance. There exists a constant $\alpha_T > 0$ such that
\[
\|\mathcal{T}(\cdot | x, u) - \mathcal{T}(\cdot | x', u)\|_{\text{TV}} \leq \alpha_T \|x - x'\|_1, \quad \forall x, x' \in \mathds{X}, \forall u \in \mathds{U}.
\]
\end{itemize}
\end{assumption}
Under these assumptions, we first recall the following theorem:

\begin{theorem}[Kara et al., 2023, Theorem 3 \cite{kara2023qlearning}]\label{theorem:error_bound_discounted}
Under Assumptions~\ref{assumption:model} and \ref{assumption:lipschitz_cost_and_transitions}, for any initial state $x_0 \in \mathds{X}$, the error between the optimal value function $J^*_{\beta}(x_0)$ and the approximate value function $\hat{J}_{\beta}(x_0)$ satisfies:
\begin{equation*}
\left| \hat{J}_{\beta}(x_0) - J^*_{\beta}(x_0) \right| \leq \left( \alpha_c + \frac{\beta \alpha_T \|c\|_{\infty}}{1 - \beta} \right) \sum_{t=0}^{\infty} \beta^t \sup_{\gamma \in \Gamma} \mathbb{E}_{x_0}^{\gamma} [L(X_t)],
\end{equation*}
where $\Gamma$ is the set of admissible policies, and $L(X_t)$ is the loss function defined in \eqref{eq:loss_function}.
\end{theorem}

The bound in Theorem~\ref{theorem:error_bound_discounted} involves a supremum over all admissible policies $\Gamma$, which can be difficult to compute. In the following, we refine the bound by restricting the supremum to stationary policies, which will turn out to be consequential in our analysis to follow as this will allow the bounds to be computed in terms of occupation measures or invariant measures.

\begin{theorem}\label{theorem:refined_error_bound}
Under Assumptions~\ref{assumption:model} and \ref{assumption:lipschitz_cost_and_transitions}, for any initial state $x_0 \in \mathds{X}$, the error satisfies:
\begin{equation*}
\left| \hat{J}_{\beta}(x_0) - J^*_{\beta}(x_0) \right| \leq \left( \alpha_c + \frac{\beta \alpha_T \|c\|_{\infty}}{1 - \beta} \right) 
 \mathbb{E}_{x_0 }^{\gamma_s} \left[ \sum_{t=0}^{\infty} \beta^t L(X_t) \right], 
\end{equation*} 
where $\gamma_s$ is the policy that achieves the supremum for $\sup_u \int |\hat{J}_\beta(x_1)-{J}^*_\beta (x_1)|\mathcal{T}(dx_1|x,u)$.

In particular, we have that
\begin{equation}
\left| \hat{J}_{\beta}(x_0) - J^*_{\beta}(x_0) \right| \leq \left( \alpha_c + \frac{\beta \alpha_T \|c\|_{\infty}}{1 - \beta} \right) 
\sup_{\gamma_s \in \Gamma_s} \mathbb{E}_{x_0 }^{\gamma_s} \left[ \sum_{t=0}^{\infty} \beta^t L(X_t) \right], \label{eq:stationary_policy}
\end{equation} 
where $\Gamma_s$ is the set of stationary policies.
\end{theorem}

\begin{proof}
    The proof is in Appendix \ref{appendix:refined_error_bound}
\end{proof}

We introduce the \emph{discounted occupation measure}.  
For any measurable set \(D \in \mathcal{B}(\mathbb{X}\!\times\!\mathbb{U})\) we define  
\begin{align*}
  \nu_{x_0}^{\gamma_s}(D) 
  = \sum_{t=0}^{\infty} \beta^{t}\,
     \mathbb{E}_{x_0}^{\gamma_s}\!\bigl[\mathbf{1}_{D}(X_t,U_t)\bigr] = \sum_{t=0}^{\infty} \beta^{t}\,
     \mathbb{P}_{x_0}^{\gamma_s}\!\bigl((X_t,U_t)\in D\bigr),
\end{align*}
where the probability measure \(\mathbb{P}_{x_0}^{\gamma_s}\) over the state and action process is defined by the
initial condition $x_0$, the policy \(\gamma_s\), and the kernel \(\mathcal{T}\). For a product set \(A\times\mathds{U}\) with \(A\subseteq\mathds{X}\) measurable we obtain
\begin{align*}
  \nu_{x_0}^{\gamma_s}(A\times\mathds{U})
      = \sum_{t=0}^{\infty} \beta^{t}\,
         \mathbb{P}_{x_0}^{\gamma_s}\bigl(X_t\in A\bigr) = \sum_{t=0}^{\infty} \beta^{t}\, \mu_t^{\gamma_s}(A) =: \frac{1}{1-\beta}\,\mu_{\beta}^{\gamma_s}(A),
\end{align*}
where \(\mu_{\beta}^{\gamma_s}\) is a probability measure on \((\mathds{X},\mathcal{B}(\mathds{X}))\) obtained by
\[
  \mu_{\beta}^{\gamma_s}(A) := (1-\beta)\sum_{t=0}^{\infty} \beta^{t}\, \mu_t^{\gamma_s}(A).
\]

\begin{theorem} \label{theorem:optimized_error_bound}
Under Assumption \ref{assumption:model} and Assumption \ref{assumption:lipschitz_cost_and_transitions}, and given a collection of quantization bins \( \{B_i\}_{i=1}^M \), we have for any initial state \( x_0 \in \mathds{X} \):
\begin{align*}
\left| \hat{J}_{\beta}(x_0) - J^*_{\beta}(x_0) \right| &\leq \left( \alpha_c + \frac{\beta \alpha_T \|c\|_{\infty}}{1 - \beta} \right) \frac{1}{1-\beta}\sup_{\gamma_s\in\Gamma_s}\int_{\mathds{X}}L(x) \mu_{\beta}^{\gamma_s}(dx) 
\end{align*}
where \( \Gamma_s \) represents the set of stationary policies. The term \( \mu_{\beta}^{\gamma_s}(dx) \) denotes the normalized discounted occupation measure over the state space.
\end{theorem}

\begin{proof}
The proof is in Appendix \ref{appendix:optimized error bound}.
\end{proof}

%We now proceed to optimize the upper bound above. The following lemma is critical.
%To minimize the error bound obtained, we optimize the quantizer by choosing the centroids of the quantization bins properly. The following lemma will be useful in choosing the centroids. 

The following is a supporting result. 
%%are supporting observations.
%\begin{lemma}\cite{degroot1970optimal}\label{lemma:l1_centroid_general}
%Let $\mathds{X}\subseteq\mathds{R}$, and let  $X$  be a random variable with distribution  $\mu$  over a quantization bin  $B_i$. Suppose that \( \mathbb{E}(|X|) < \infty \). Then, the value  $y_i$  that minimizes the expected  distortion  $\mathbb{E}[ | X - y_i | ]$ is the median of  $X$  with respect to the probability measure $\mu$  over  $B_i$ .
%\end{lemma}

%Therefore, for a given quantization bin \( B_i \), the expected  distortion is minimized when the representative point \( y_i \) is chosen as the median of \( B_i \) with respect to the probability measure \( \mu \) over \( B_i \). This turns the normalized measure  $\hat{\pi}_{y_i}$ into the Dirac measure $\delta_{y_i}$ concentrated at $y_i$.  This result can be also applied to higher dimensional state spaces in $\mathbb{R}^n$:
%%The result can be extended to higher dimensions when considering the $\ell_1$ norm in $\mathbb{R}^n$. 

\begin{lemma}
\label{theorem:higher_dimensions}
Let $X = (X_1, X_2, \dots, X_n)$ be a $\mathds{R}^n$-valued random vector with finite expected absolute deviations in each coordinate. Then, the point $y_i^* = (y_{i,1}^*, y_{i,2}^*, \dots, y_{i,n}^*)$ that minimizes $\mathbb{E}[ \| X - y_i \|_1 ]$ is obtained by choosing each $y_{i,k}^*$ to be a median of the marginal distribution of $X_k$ over the bin $B_i$, for $k = 1, 2, \dots, n$.
\end{lemma}

The case with $n=1$ is proven in \cite{degroot1970optimal}. The generalization for the $n$-dimensional case is then immediate: Let $y_i \in \mathbb{R}^n$ . The expected  $\ell_1$  distortion writes as:
\begin{equation}
\mathbb{E}[ \| X - y_i \|_1 ] = \int_{B_i} \| x - y_i \|_1 \mu(dx) = \sum_{k=1}^n \int_{B_i} | x_k - y_{i,k} | \mu_k(dx),
\end{equation}
where $\mu_k$ is the marginal distribution of $X_k$ over the bin $B_i$. We thus can minimize each term separately and the result follows. 

%By the same reasoning as in the scalar case, as in Lemma \ref{lemma:l1_centroid_general}, $\int_{B_i} | x_k - y_{i,k} | \mu_k(dx)$ is minimized when $y_{i,k}$  is a median of $X_k$  with respect to the probablity measure $\mu_k$ over $B_i$. Therefore, the point  $y^*_i$  that minimizes the expected $\ell_1$ distortion is found as:
%\begin{equation}
%y_i^* = (y_{i,1}^*, y_{i,2}^*, \dots, y_{i,n}^*), \quad \text{for} \quad k=1,2,\dots,n
%\end{equation}
%where each $y_{i,k}^*$ is the median of the quantization bin $B_i$ with respect to the marginal distribution $\mu_k$.

\begin{corollary}[to Theorem \ref{theorem:optimized_error_bound} and Theorem \ref{theorem:higher_dimensions}] \label{median_bound_disc}
Under Assumption \ref{assumption:model} and Assumption \ref{assumption:lipschitz_cost_and_transitions}, and given a collection of quantization bins \( \{B_i\}_{i=1}^M \), we have for any initial state \( x_0 \in \mathds{X} \):
\begin{equation*}
\left| \hat{J}_{\beta}(x_0) - J^*_{\beta}(x_0) \right| \leq \left( \alpha_c + \frac{\beta \alpha_T \|c\|_{\infty}}{1 - \beta} \right) 
\sup_{\gamma_s \in \Gamma_s} \frac{1}{1 - \beta} \sum_{i=1}^{M} \int_{B_i} \|x - y_i\|_1 \mu_{\beta}^{\gamma_s}(dx),
\end{equation*} 
where \( \Gamma_s \) represents the set of stationary policies, and \( y_i \) is the median of the quantization bin \( B_i \) under the measure $\mu_{\beta}^{\gamma_s}(dx) $ where \( \mu_{\beta}^{\gamma_s}(dx) \) denotes the normalized discounted occupation measure over the state space.
\end{corollary}

%\subsection{Refined Bounds Using Lyapunov Stability}

Our following result illustrates how the expected loss during the application of a quantization can be bounded by the use of a Lyapunov function for non-compact state spaces. 

\begin{theorem}\label{theorem:uniform_quantizer_lyapunov}

Let $\mathds{X}\subseteq\mathbb{R}^{n}$, $b\ge 0$, $\alpha>0$, and define the Lyapunov function 
\(
  V(x)=\|x\|_{1}^{m}.
\)
 Assume the controlled process $\{X_t\}$ satisfies the drift condition
\[
\label{eq:lyapunov_eq_nd}
   \mathbb{E}\bigl[V(X_{t+1})\mid X_t=x,U_t=u\bigr]
   \le V(x)-\alpha V(x)+b,
   \qquad x\in\mathds{X},u\in\mathds{U}
   \]
  Under Assumption~\ref{assumption:model} and 
Assumption~\ref{assumption:lipschitz_cost_and_transitions},
let $M$ be the total number of hyper-cubic bins in the uniform
quantizer.  
Then, for every initial state $x_{0}\in\mathds{X}$ we have that,
\begin{equation}
   \bigl|\hat J_{\beta}(x_{0})-J^{*}_{\beta}(x_{0})\bigr|
   \le \left( \alpha_c + \frac{ \beta \alpha_T \| c \|_{\infty} }{ 1 - \beta } \right) \frac{ (2n+1) C^{1/m} }{ (M^{ \frac{1}{n}(1-\frac{1}{m}) }) (1-\beta) },
   \label{eq:error_bound_nd}
\end{equation}
where
\[
   C:=\frac{\|x_{0}\|_{1}^{m}(1-\beta)+b\beta}{1-\beta(1-\alpha)} .
\]
%Assume that the state space $\mathds{X} \subseteq \mathbb{R} $ and let $b \geq 0$, $\alpha >0$ and $V:\mathds{X} \to[0,\infty)$. Assume the state process $\{X_t\}$ satisfies the following condition:
%\begin{equation}
%\mathbb{E}^{\gamma_s}[V(X_{t+1})|X_t=x,U_t=u] \leq V(x) - \alpha V(x) + b, \quad x \in \mathds{X},u \in \mathds{U}\label{eq:lyapunov_eq}
%\end{equation}
%for all $\gamma_s\in\Gamma_s$ where $V(x)=|x|^m$, for some $m>1$.
%
%Then, under Assumption~\ref{assumption:model} and Assumption~\ref{assumption:lipschitz_cost_and_transitions}, with $M$ denoting the number of bins, there exists a quantizer which leads to the following error bound for any initial state \( x_0 \in \mathds{X}\):
%\begin{equation*}
%\left| \hat{J}_{\beta}(x_0) - J^*_{\beta}(x_0) \right| \leq \left( \alpha_c + \frac{ \beta \alpha_T \| c \|_{\infty} }{ 1 - \beta } \right) \frac{ 3 C^{1/m} }{ (M^{ 1 - 1/m }) (1-\beta) },
%\end{equation*}
%where $C$ is a constant, which depends on the initial state $x_0$, defined as: 
%\begin{equation*}
%C:= \frac{|x_0|^m(1-\beta)+b\beta}{1-\beta(1-\alpha)}
%\end{equation*}
\end{theorem}

\begin{proof}
    The proof is in Appendix~\ref{appendix:uniform_quantizer_lyapunov}.
\end{proof}

We note that the existence discussion is constructive and the proof of Theorem \ref{theorem:uniform_quantizer_lyapunov} utilizes an explicit quantizer which attains the bound presented. Observe that as $M \to \infty$, the error converges to zero.

%The next result extends Theorem~\ref{theorem:uniform_quantizer_lyapunov} to an $n$–dimensional state space $\mathds{X}\subseteq\mathbb{R}^{n}$ (with the $\ell_{1}$–norm).

%\begin{theorem}\label{theorem:uniform_quantizer_lyapunov_nd}
%Let $\mathds{X}\subseteq\mathbb{R}^{n}$, $b\ge 0$, $\alpha>0$, and define the Lyapunov function 
%\(
%  V(x)=\|x\|_{1}^{m}
%\)
%with exponent $m>n$.  Assume the controlled process $\{X_t\}$ satisfies the drift condition
%\[
%\label{eq:lyapunov_eq_nd}
%   \mathbb{E}\bigl[V(X_{t+1})\mid X_t=x,U_t=u\bigr]
%   \le V(x)-\alpha V(x)+b,
%   \qquad x\in\mathds{X},u\in\mathds{U}
%   \]
%  Under Assumption~\ref{assumption:model} and 
%Assumption~\ref{assumption:lipschitz_cost_and_transitions},
%let $M$ be the total number of hyper-cubic bins in the uniform
%quantizer.  
%Then, for every initial state $x_{0}\in\mathds{X}$ we have that,
%\begin{equation}
%   \bigl|\hat J_{\beta}(x_{0})-J^{*}_{\beta}(x_{0})\bigr|
%   \le\Bigl(\alpha_{c}+\frac{\beta\alpha_{T}\|c\|_{\infty}}{1-\beta}\Bigr)\frac{(2n+1)C^{1/m}}{M^{\frac{1}{n}-\frac{1}{m}}(1-\beta)},
%   \label{eq:error_bound_nd}
%\end{equation}
%where
%\[
%   C:=\frac{\|x_{0}\|_{1}^{m}(1-\beta)+b\beta}{1-\beta(1-\alpha)} .
%\]
%\end{theorem}

%\begin{proof}
%The proof is in Appendix~\ref{appendix:uniform_quantizer_lyapunov_n}. 
%\end{proof}

%The explicit quantizer is presented in the proof of Theorem \ref{theorem:uniform_quantizer_lyapunov_nd}.

\subsection{Refined Error Bounds for the Average Cost Criterion}
\label{section:refined_average}

In this section, we extend the analysis to the average cost criterion in Markov Decision Processes (MDPs). 
We focus on the long-run average cost criterion, defined as:
\[
J_{\text{avg}}(x_0,\gamma) = \limsup_{T \to \infty} \frac{1}{T} \mathbb{E}^{\gamma}_{x_0} \left[ \sum_{t=0}^{T-1} c(X_t, U_t) \right],
\]
where \( \gamma = \{\gamma_t\}_{t=0}^\infty \) is the policy, and \( J_{\text{avg}}(x_0) \) represents the average cost starting from state \( x_0 \).

We now define the average cost problem as follows:
\[
J_{\text{avg}}^*(x) := \inf_{\gamma} J_{\text{avg}}(x, \gamma) = \inf_{\gamma \in \Gamma_A} \limsup_{T \to \infty} \frac{1}{T} \mathbb{E}_x^\gamma \left[ \sum_{t=0}^{T-1} c(X_t, U_t) \right]
\]
where \( \Gamma_A \) represents the set of all admissible policies.

\begin{assumption}[Minorization Condition]\label{assumption:minorization}
There exists a non-trivial positive measure $\mu$ on $\mathds{X}$ such that for all $(x, u) \in \mathds{X} \times \mathds{U}$:
\begin{equation}
\mathcal{T}(B \mid x, u) \geq \mu(B), \quad \forall B \in \mathcal{B}(\mathds{X}).
\label{eq:minorization}
\end{equation}
\end{assumption}

Before the result, we introduce the Average Cost Optimality Equations (ACOEs) for the original and the finite model:
\begin{align*}
   & h(x) = \inf_{u \in \mathds{U}} \left\{ c(x,u) + \int_{\mathds{X}} h(x_1) \mathcal{T}(dx_1 | x,u) \right\} - \int_{\mathds{X}} h(x) \mu(dx)\\
&\hat{h}(y) = \inf_{u \in \mathds{U}} \left\{ C^*(y,u) + \sum_{y_1} \hat{h}(y_1) P^*(y_1 | y,u) \right\} - \int_{\mathds{X}} \hat{h}(q(x)) \mu(dx).
\end{align*}
Existence of the solutions to these equations is guaranteed under  Assumption \ref{assumption:minorization}. We will refer to the functions $h$ and $\hat{h}$ as the relative value functions in the following. 
\begin{theorem}
\label{theorem:error_bound_average_alt}
Under Assumptions \ref{assumption:model}, \ref{assumption:lipschitz_cost_and_transitions} and  \ref{assumption:minorization}, $J_{\text{avg}}^*(x)$ and $\hat{J}_{\text{avg}}(x)$ are constants for any $x \in \mathds{X}$:
\begin{align*}
    &j^* = J_{\text{avg}}^*(x) \quad \text{for all } x \in \mathds{X}, \\
    &\hat{j} = \hat{J}_{\text{avg}}(x) \quad \text{for all } x \in \mathds{X}
\end{align*}
    Moreover, 
\begin{equation}
| J_{\text{avg}}^*(x_0) - \hat{J}_{\text{avg}}(x_0) | \leq \left(\alpha_c+ \frac{\alpha_T \|c\|_\infty}{\mu(\mathds{X})} \right) \int_{\mathds{X}} L(x) \pi_{\gamma_s}(dx),
\end{equation}
where $\gamma_s$ is the policy that achieves the supremum for $\sup_u\int |h(x_1)-\hat{h}(x_1)|\mathcal{T}(dx_1|x,u)$
and where  $\pi_{\gamma_s}$ is the invariant measure induced by policy $\gamma_s$.
\end{theorem}

\begin{proof}
    The proof is in Appendix \ref{appendix:error_bound_average_alt}
\end{proof}

%Now, we state the following theorem, which will be useful in the proof of our next result:
%
%\begin{theorem}[Comparison Theorem] \cite[Theorem 14.2.2]{meyn1993markov}\label{theorem:comparison_theorem}
%Let \( V:\mathds{X} \to[0,\infty) \) and \( f, g : \mathds{X} \to [0,\infty) \). Let \( \{x_n\} \) be a Markov chain on \( \mathds{X} \). If the following condition is satisfied:
%\[
%\int_{\mathds{X}} P(x, dy) V(y) \leq V(x) - f(x) + g(x), \quad x \in \mathds{X},
%\]
%then, for any stopping time \( \tau \) with \( \mathbb{P}(\tau < \infty) = 1 \), it holds that
%\[
%\mathbb{E}\left[\sum_{t=0}^{\tau-1} f(x_t)\right] \leq V(x_0) + \mathbb{E}\left[\sum_{t=0}^{\tau-1} g(x_t)\right].
%\]
%\end{theorem}
%
%\begin{proof}
%See \cite{yuksel2024lecture}.
%\end{proof}
%
%Now, we present our result under Lyapunov condition similar to the quantizer design for discounted cost criterion.

\begin{theorem}\label{theorem:uniform_quantizer_lyapunov_average}
Assume $\mathds{X}\subseteq\mathbb{R}^n$, $f(x)=\|x\|_1^m$ with $m>1$, and suppose
\begin{align}\label{eq:lyapunov_eq_avg}
\sup_{x,u}\mathbb{E}[V(X_{t+1})\mid X_t=x,U_t=u]\le V(x)-f(x)+b,
\end{align}
where $b\ge 0$ and $V:\mathds{X}\to[0,\infty)$ is a Lyapunov function.

Then, under Assumption~\ref{assumption:model}, Assumption~\ref{assumption:lipschitz_cost_and_transitions}  and Assumption~\ref{assumption:minorization}, with $M$ denoting the number of bins, there exists a quantizer which leads to the following error bound for any initial state \( x_0 \in \mathds{X}\):
\begin{equation}
|\hat{J}_{avg}(x_0)-J_{avg}^*(x_0)| \le  \left(\alpha_c+ \frac{\alpha_T \|c\|_\infty}{\mu(\mathds{X})} \right)\frac{(2n+1)b^{1/m}}{M^{1/n(1-1/m)}}.
\end{equation}
\end{theorem}

\begin{proof}
    The proof is in Appendix~\ref{appendix:uniform_quantizer_lyapunov_average}
\end{proof}

As earlier, as $M \to \infty$ the error converges to zero.

%\sy{Ali: the proof can be applied to spaces beyond $\mathbb{R}^n$, such as probability measure valued spaces by replacing the $L_1$ distance with $W_1$ distance. Should we add a note on this? Though. the Lyapunov analysis would be much more complicated...}

\section{Quantizer Design for Quantized Q-Learning and Empirical Model Learning}

In this section, we study the design of quantizers for the Q-learning algorithm, and via equivalence, to empirical model learning. To this end, we first refine the error expression obtained in  \cite{kara2021convergence}, and obtain an explicit error bound by extending the analysis presented in Section~\ref{section:refined_discounted}, and also design quantizers applicable for unbounded state spaces.

\subsection{Quantized Q-Learning Algorithm and its Convergence}

As noted in \cite{kara2021convergence}, quantizing the state space of a continuous MDP converts the problem into a Partially Observable Markov Decision Process (POMDP) which is then non-Markovian. Let $q$ be a quantizer mapping $\mathbb{X}$ to a finite set as described in Section \ref{finiteModelAppC}.

%\[
%Y_t = q(X_t),
%\]
%where \( q: \mathbb{X} \rightarrow \mathbb{Y} \) is the quantizer, which maps the continuous state space \( \mathbb{X} \) to the finite set of representative points \( \mathbb{Y} = \{y_1, y_2, \ldots, y_M\} \).
%
%The observation model in our setup is described as:
%\[
%O(y_i \mid x) = \mathbb{P}(Y_t = y_i \mid X_t = x) = \mathds{1}_{B_i}(x),
%\]
%where \( \mathds{1}_{B_i}(x) \) is the indicator function. In this model, the agent has access to only the representative points $Y_t = y_i$, which indicates that the true state $X_t$ resides in the quantization bin $B_i$. In other words, the true state $X_t$ remains hidden from the agent. In the quantized MDP, the agent observes $Y_t = q(X_t)$. 
%In a single sample path, the agent collects realizations of states, actions and stage-wise cost under the exploration policy:
%\begin{equation}
%X_0, U_0, c(X_0, U_0), X_1, U_1, c(X_1, U_1), \ldots
%\end{equation}

Now, consider the following Q-learning update rule for $(X_T,U_T)=(x,u) \in \mathbb{X} \times \mathbb{U}$:
\begin{equation}\label{eq:quantized_q_learning}
Q_{t+1}( q(x), u ) = (1 - \alpha_t( q(x), u )) Q_t( q(x), u ) + \alpha_t( q(x), u ) \left( c( x, u ) + \beta \min_{v \in \mathbb{U}} Q_t( q(X_{t+1}), v ) \right),
\end{equation}
where, $\alpha_t( y_t, u_t )$ is the learning rate and $c( x, u )$ is the immediate cost which depends on the true state $x$ and action $u$. The quantized Q-learning algorithm is then implemented as:
\begin{algorithm}
\caption{Quantized Q-Learning Algorithm}
\label{alg:quantized_q_learning}
\begin{algorithmic}
\STATE \textbf{Input:} Initial Q-function $Q_0$, quantizer $q: \mathbb{X} \to \mathbb{Y}$, exploration policy $\gamma^*$, total iterations $L$
\STATE Initialize counts $N(y, u) = 0$ for all $(y, u) \in \mathbb{Y} \times \mathbb{U}$
\FOR{$t = 0$ to $L - 1$}
    \STATE Observe the state $X_t$ and quantize the state according to $y_t = q( X_t )$
    \STATE Select action $u_t$ according to the exploration policy $\gamma^*$
    \STATE Execute action $u_t$, receive cost $c( X_t, u_t )$, observe next state $X_{t+1}$
    \STATE Observe next quantized state $y_{t+1} = q( X_{t+1} )$
    \STATE Update the count: $N( y_t, u_t ) \leftarrow N( y_t, u_t ) + 1$
    \STATE Update the learning rate:
    \begin{equation}
    \alpha_t( y_t, u_t ) = \frac{1}{1 + N( y_t, u_t ) }
    \end{equation}
    \STATE Update the Q-function:
    \begin{equation}\label{eq:q_update}
    Q_{t+1}( y_t, u_t ) = (1 - \alpha_t( y_t, u_t )) Q_t( y_t, u_t ) + \alpha_t( y_t, u_t ) \left( c( X_t, u_t ) + \beta \min_{v \in \mathbb{U}} Q_t( y_{t+1}, v ) \right)
    \end{equation}
\ENDFOR
\STATE \textbf{Output:} Learned Q-function $Q_L$
\end{algorithmic}
\end{algorithm}

\begin{assumption}\label{assumption:pomdp_convergence}
    \begin{enumerate}
    \item We define the step size \( \alpha_t(y, u) \) as follows:
    \[
    \alpha_t(y, u) = 
    \begin{cases} 
      0 & \text{if } (Y_t, U_t) \neq (y, u), \\
      \frac{1}{1 + \sum_{k=0}^{t} \mathds{1}\{Y_k = y, U_k = u\}} & \text{otherwise}.
    \end{cases}
    \]
    \item Under the exploration policy \( \gamma^* \), the state process \( \{X_t\}_{t \geq 0} \) is uniquely ergodic, implying the existence of a unique invariant measure \( \pi_{\gamma^*} \).
    \item During the exploration phase, each observation-action pair \( (y, u) \) is visited infinitely often. This ensures sufficient exploration of the state-action space.
\end{enumerate}
\end{assumption}

We note that a sufficient condition for the second item above is that the state process $\{ X_t \}_{t \geq 0}$ is positive Harris recurrent. We have the following convergence result:

\begin{theorem}[Theorem 9 \cite{kara2023qlearning}]\label{theorem:convergence_quantized_q_learning}
Under Assumption~\ref{assumption:pomdp_convergence}, the quantized Q-learning algorithm in \ref{alg:quantized_q_learning} converges almost surely to a function $Q^*( y, u )$ that satisfies the following fixed-point equation for every $(y,u) \in \mathbb{Y} \times \mathbb{U} $:
\begin{equation}\label{eq:quantized_q_fixed_point}
Q^*( y, u ) = C^*( y, u ) + \beta \sum_{ y' \in \mathbb{Y} } P^*( y' | y, u ) \min_{ v \in \mathbb{U} } Q^*( y', v ),
\end{equation}
where:
\begin{align}
C^*( y, u ) &= \frac{1}{ \pi_{\gamma^*}( B_y ) } \int_{ B_y } c( x, u ) \pi_{\gamma^*}( dx ), \label{eq:effective_cost_quantized} \\
P^*( y' | y, u ) &= \frac{1}{ \pi_{\gamma^*}( B_y ) } \int_{ B_y } \int_{ B_{ y' } } \mathcal{T}( x' | x, u ) \, dx' \, \pi_{\gamma^*}( dx ), \label{eq:effective_transition_quantized}
\end{align}
where, $B_y$ denotes the quantization bin corresponding to $y \in \mathbb{Y}$ and $\pi_{\gamma^*}$ is the invariant measure of the state process under the exploration policy $\gamma^*$:
\end{theorem}

%\begin{proof}
%The proof follows from viewing quantized MDP as a POMDP and applying Theorem \ref{theorem:q_learning_convergence} directly.
%\end{proof}

\subsection{Empirical Model Learning and Equivalence with Quantized Q-Learning}

Let under the exploration policy $\gamma^*$ given in the quantized Q-learning algorithm in \ref{alg:quantized_q_learning} give rise to the invariant probability measure $\pi_{\gamma^*}$.
The limiting Q-function $Q^*( y, u )$ in the discussion above corresponds to the optimal Q-function of an approximate MDP defined over the quantized state space $\mathbb{Y}$. The effective cost $C^*( y, u )$ is the average cost over the bin $B_y$ weighted by the invariant distribution $\pi_{\gamma^*}$ conditioned on bin $B_y$:
\begin{equation}\label{modelApp}
C^*( y, u ) = \mathbb{E}_{ x \sim \pi_{\gamma^*} | x \in B_y } [ c( x, u ) ] = \int_{B_y} \frac{\pi_{\gamma^*}(dx)}{\pi_{\gamma^*}(B_y)}c(x,u) .
\end{equation}

Observe that the above is, see e.g. \cite[Theorem 2.1]{karayukselNonMarkovian}, equal to the almost sure limit of the empirical expression on the right hand side below:
\begin{equation}\label{empCostEst}
C^*( y, u )  = \lim_{N \to \infty} \frac{\sum_{k=0}^{N-1} c(X_k,U_k) 1_{\{X_k \in B_y,U_k=u\}}}{ \sum_{k=0}^{N-1} 1_{\{X_k \in B_y,U_k=u\}}} 
\end{equation}

Similarly, the effective transition probability $P^*( y' | y, u )$ represents the probability of transitioning from bin $B_y$ to bin $B_{ y' }$ under action $u$, averaged over the invariant distribution:
\begin{equation}\label{costApp}
P^*( y' | y, u ) = \mathbb{P}_{ x \sim \pi_{\gamma^*} | x \in B_y } [ q( X_{ t+1 } ) = y' | X_t = x, U_t = u ] =  \int_{B_y} \frac{\pi_{\gamma^*}(dx)}{\pi_{\gamma^*}(B_y)} {\cal T}(B_{y'} | x,u).
\end{equation}

Likewise, by \cite[Theorem 2.1]{karayukselNonMarkovian}, the above is the almost sure empirical limit of of the right hand side below:
\begin{equation}\label{empTranEst}
P^*( y' | y, u ) = \lim_{N \to \infty} \frac{\sum_{k=0}^{N-1} 1_{\{X_{k+1} \in B_{y'}\} } 1_{\{X_k \in B_y,U_k=u\}}}{ \sum_{k=0}^{N-1} 1_{\{X_k \in B_y,U_k=u\}}} 
\end{equation}

An interpretation of the above result then is that one can first obtain the approximate model given with (\ref{modelApp}-\ref{costApp}) by forcing the data into a Markovian model for both the empirical cost estimate (\ref{empCostEst}) and empirical transition kernel estimate (\ref{empTranEst}), and then solve the MDP as if this empirically constructed model is the actual one, instead of running Q-learning whose limit is then optimal precisely for this learned/empirically constructed model. A benefit of such a model-based approach is that one can have better sample complexity bounds compared with Q-learning for certain applications, see e.g. \cite{zhou2024robustness} (see also \cite[Section 5.1]{kara2020robustness}). 

\begin{remark} In the planning framework presented in the previous section, we had the flexibility to select the weighting measures \( \pi_{y_i} \) over the quantization bins arbitrarily. This allowed us to minimize the expected loss by choosing \( \pi_{y_i} \) as a Dirac measure centered at the median of each bin. In contrast, within the learning context, the weighting measure \( \pi_{y_i} \) is dictated by the exploration policy 
and is inherently dependent on the structure of the quantization bins \( \{ B_i \} \). %Specifically, \( \pi_{y_i} \) corresponds to the normalized invariant distribution over bin \( B_i \) induced by the exploration policy; that is, it is the conditional probability measure given that the state is in $B_i$: 
%\[ \pi_{y_i}(A) = \frac{1_{\{A \in B_i\}}\pi_{\gamma^*}( A )}{\pi_{\gamma^*}( B_y )}, \qquad A \in B_i, A \in {\cal B}(\mathbb{X})\]
%
\end{remark}

\subsection{Error Analysis and Quantizer Design for Model Learning and Quantized Q-Learning}

As noted above, the weighting measure depends on the exploration policy, quantizer and system model, and cannot be assigned arbitrarily. This dependence introduces a key difficulty in bounding the expected loss, as the weighing measure over the overflow bin, i.e. \( \pi_{y_{M+1}} \) can no longer be freely chosen or controlled. This limitation is significant for the analysis of non-compact spaces. %Nevertheless, we will show that it is still possible to establish a meaningful upper bound on the expected error by employing a Lyapunov-based analysis, analogous to the approach presented in Section~\ref{section:refined_discounted}.

In the previous section, we studied the approximate model where the weighting measures \( \hat{\pi}^*_{y_i} \) for each quantization bin were chosen freely and \( \hat{\pi}^*_{y_i} \) were chosen as Dirac measures concentrated at the \( \ell_1 \) centroids (medians) of each bin \( B_i \). However, when implementing Q-learning in the POMDP framework as described in this section, we lose the freedom to choose these weighting measures independently. The measures \( \hat{\pi}^*_{y_i} \)'s are inherently determined by the invariant distribution \( \pi_{\gamma^*} \) under the exploration policy \( \gamma^* \). We first start with the following result which relates the approximation error to occupation measures. 

\begin{theorem}\label{theorem:error_quantized_q_learning}
Under Assumption Assumption~\ref{assumption:lipschitz_cost_and_transitions} and Assumption \ref{assumption:pomdp_convergence}, given a collection of quantization bins \( \{B_i\}_{i=1}^M \), we have for any initial state $x_0$:
\begin{align}\label{refinedExpLearnL}
 \left| J^*_{\beta}(x_0) - \min_v \{ Q^*(x_0, v) \} \right|  \leq \left( \alpha_c + \frac{\beta \alpha_T \|c\|_{\infty}}{1 - \beta} \right) 
\sup_{\gamma_s \in \Gamma_s} \frac{1}{1 - \beta} \sum_{i=1}^{M} \int_{B_i} \int_{B_i} \|x - x'\|  \mu_{\beta}(dx) \hat{\pi}_{y_i}(dx'),
\end{align}
where, \( \Gamma_s \) represents the set of stationary policies, \( \mu_{\beta} \) is the discounted occupation measure under the policy \( \gamma_s \) and \( \hat{\pi}_{y_i}(dx') \) is the normalized invariant measure obtained under the exploration policy \( \gamma^* \).
\end{theorem}

\begin{proof}
%We begin by noting that, by (Kara and Yüksel, 2021, Theorem 4.1 \cite{kara2023convergence}), the Q-learning algorithm converges to the optimal Q-values, such that \( \min_v\{Q^*(x,v)\} = \hat{J}_\beta(x) \). Thus, we can use the optimal Q-function to obtain the value function for the approximate model.

From Theorem \ref{theorem:refined_error_bound}, we have the error bound:
\begin{equation}
\left| \hat{J}_{\beta}(x_0) - J^*_{\beta}(x_0) \right| \leq \left( \alpha_c + \frac{\beta \alpha_T \|c\|_{\infty}}{1 - \beta} \right) 
\sup_{\gamma_s \in \Gamma_s} \mathbb{E}_{x_0 }^{\gamma_s} \left[ \sum_{t=0}^{\infty} \beta^t L(X_t) \right],
\end{equation}
where \( \Gamma_s \) represents the stationary policies, and \( L(X_t) \) is defined in \eqref{eq:loss_function}. As in the analysis leading to \eqref{eq:expected_loss}, substituting the expected accumulated discounted loss function, the bound is refined to (\ref{refinedExpLearnL}), where \( \hat{\pi}_{y_i}(dx') \) is not freely chosen but obtained from the invariant probability measure under the exploration policy. 
\end{proof}

In the following, we obtain an explicit bound which demonstrates the applicability of quantized Q-learning for non-compact spaces for quantized Q-learning.

\begin{theorem}\label{theorem:lyapunov_learning}
    Assume that the state space $\mathds{X} \subseteq \mathbb{R}^n $ and let $b \geq 0$, $V:\mathds{X} \to[0,\infty)$, $f:\mathds{X} \to [\epsilon,\infty)$ for some $\epsilon>0$. Assume the state process $\{X_t\}$ satisfies the following condition:
\begin{equation}
\sup_{x\in\mathds{X},u\in\mathds{U}} \mathbb{E}[V(X_{t+1})|X_t=x,U_t=u] \leq V(x) - \alpha V(x) + b,\label{eq:lyapunov_eq_learning}
\end{equation}
where $V(x)=\|x\|_1^m$, $m>1$.
Then, under Assumption \ref{assumption:model} and Assumption \ref{assumption:lipschitz_cost_and_transitions}, provided that the cost function \( c \) is bounded, with $M$ denoting the number of bins, there exists a quantizer which leads to the following error bound for any initial state \( x_0 \in \mathbb{X}\):
\begin{equation}
\left| \hat{J}_{\beta}(x_0) - J^*_{\beta}(x_0) \right| \leq \left( \alpha_c + \frac{ \beta \alpha_T \| c \|_{\infty} }{ 1 - \beta } \right) \left(\frac{ 4 C^{1/m} }{ (M^{1/n( 1 - 1/m )}) (1-\beta) }\right),
\end{equation}
where $C$ is a constant, which depends on the initial state $x_0$, defined as: 
\[
C:= \frac{\|x_0\|_1^m(1-\beta)+b\beta}{1-\beta(1-\alpha)}.
\]
\end{theorem}
\begin{proof}
The proof is in Appendix \ref{proof:lyapunov_learning}.
\end{proof}

\section{Concluding Remarks}

For MDPs with unbounded state spaces we presented upper bounds on finite model approximation errors via optimizing the quantizers used for finite model approximations. We also considered implications on quantized Q-learning and the performance of policies obtained via Q-learning where the quantized state is treated as the state itself. We noted the distinctions between planning, where approximating MDPs can be independently designed, and learning, where approximating MDPs are restricted to be defined by invariant measures of Markov chains under exploration policies, leading to significant subtleties on quantizer design performance. Nonetheless, asymptotic near optimality can be established under both setups with explicit convergence rates. %In particular, under Lyapunov growth conditions, we obtained explicit upper bounds which decay to zero as the number of bins approaches infinity. %It is currently of interest to obtain a learning algorithm which would learn the optimal quantizers.

We note that, due to relative clarity in presentation especially involving the associated Lyapunov analysis, while we studied the case with the state space being $\mathbb{R}^n$, the analysis can be directly generalized to any normed space by adopting the required regularity conditions on the kernels and cost. Notably, if one applies the analysis here to a filter-reduced MDP (known as belief-MDP) of Partially Observable Markov Decision Processes (POMDPs) (see \cite{tutorialkara2024partially} for conditions on the necessary continuity properties), by replacing $\|\cdot\|_1$ with the Wasserstein distance of order-1 on probability measures, the analysis can be applied identically.  
\appendix

\section{Proof of Theorem \ref{theorem:refined_error_bound}} \label{appendix:refined_error_bound}

\begin{proof}
We begin with the following initial bound, as in \cite{kara2023qlearning}, using the corresponding Bellman equations:
\begin{align}
\left| \hat{J}_{\beta}(x_0) - J^*_{\beta}(x_0) \right| \leq 
\left( \alpha_c + \beta \|\hat{J}_{\beta}\|_{\infty} \alpha_T \right) L(x_0) + \nonumber \\
\quad \beta \sup_{u_0 \in \mathbb{U}} 
\mathbb{E}^\gamma \left[ \left| J^*_{\beta}(X_1) - \hat{J}_{\beta}(X_1) \right| \big| x_0,u_0\right],
\end{align}
where \( \alpha_c \) is the Lipschitz constant for the cost function, \( \alpha_T \) is the Lipschitz constant for the transition kernel, and \( L(x_0) \) is the quantization error at \( x_0 \).

Let \( V(x) = |J^*_{\beta}(x) - \hat{J}_{\beta}(x)| \) represent the difference between the optimal value function and the quantized value function. Then we
have the following bound for \( V(x_0) \):
\begin{equation*}
V(x_0) \leq \left( \alpha_c + \beta \|\hat{J}_{\beta}\|_{\infty} \alpha_T \right) L(x_0) + \beta \sup_{u_0 \in \mathbb{U}} \mathbb{E} \left[ V(X_1) \mid x_0, u_0 \right].
\end{equation*}
One can show that the term $ \mathbb{E} \left[ V(X_1) \mid x_0, u_0 \right]$ when considered as a function of $x_0,u_0$ satisfies the measurable selection conditions under Assumption \ref{assumption:lipschitz_cost_and_transitions}. We denote by $f(x)$ the control function which achieves the supremum such that $$\sup_{u_0} \mathbb{E} \left[ V(X_1) \mid x_0, u_0 \right] = \mathbb{E} \left[ V(X_1) \mid x_0, f(x_0) \right]. $$

 Iterating the initial inequality for subsequent time step, we can write
\begin{align*}
V(x_0) &\leq \left( \alpha_c + \beta \|\hat{J}_{\beta}\|_{\infty} \alpha_T \right) \left( L(x_0) + \beta \mathbb{E}_{x_0} \left[ L(X_1) |x_0,f(x_0)\right] \right) \nonumber \\
&\quad + \beta^2 \mathbb{E}\left[\sup_{u_1}\mathbb{E} \left[ V(X_2)|X_1,u_1 \right]|x_0,f(x_0)\right]
\end{align*}
note that the supremum is achieved by the same control function $f$. Hence, defining the stationary policy \( \gamma_s = \{f, f, f, \dots \} \), and repeating this process up to time step \( T-1 \), we have:
\begin{equation*}
V(x_0) \leq \left( \alpha_c + \beta \|\hat{J}_{\beta}\|_{\infty} \alpha_T \right) \mathbb{E}^{\gamma_s}_{x_0} \left[ \sum_{t=0}^{T-1} \beta^t L(X_t) \right] + \beta^T \mathbb{E}^{\gamma_s}_{x_0} \left[ V(X_T) \right].
\end{equation*}
Since the cost function \( c \) is bounded, it follows that \( V(x) \) is bounded as well. Thus, \( \lim_{T \to \infty} \beta^T \mathbb{E}^{\gamma_s}_{x_0} \left[ V(X_T) \right] = 0 \), which means the second term vanishes as \( T \to \infty \). Taking the limit as \( T \to \infty \), we get:
\begin{equation*}
V(x_0) \leq \left( \alpha_c + \beta \|\hat{J}_{\beta}\|_{\infty} \alpha_T \right) \mathbb{E}^{\gamma_s}_{x_0} \left[ \sum_{t=0}^{\infty} \beta^t L(X_t) \right].
\end{equation*}
Finally, by taking the supremum over all stationary policies \( \gamma_s \), we obtain the upper bound:
\begin{equation*}
\left| \hat{J}_{\beta}(x_0) - J^*_{\beta}(x_0) \right| \leq  \left( \alpha_c + \beta \|\hat{J}_{\beta}\|_{\infty} \alpha_T \right) \sup_{\gamma_s \in \Gamma_s}\mathbb{E}^{\gamma_s}_{x_0} \left[ \sum_{t=0}^{\infty} \beta^t L(X_t) \right].
\end{equation*}
The proof is completed by noting that the term \( \|\hat{J}_{\beta}\|_{\infty} \) is bounded by \( \frac{\|c\|_{\infty}}{1 - \beta} \), due to the boundedness of the cost function \( c \).
\end{proof}

\section{Proof of Theorem \ref{theorem:optimized_error_bound}}
\label{appendix:optimized error bound}
\begin{proof}
The expected loss function \( \mathbb{E}^{\gamma_s}[L(X_t)] \) can be written as:
\begin{equation*}
\mathbb{E}^{\gamma_s}[L(X_t)] = \sum_{i=1}^{M} \int_{B_i} L(x) \mu_t(dx),
\end{equation*}
where \( \mu_t(dx) \) is the distribution of the state \( X_t \) at time \( t \) which can also be seen as the marginal of the strategic measure $\mathds{P}^{\gamma_s}_{x_0}$ over $X_t$, and the summation is over all quantization bins \( B_i \). 

Next, we substitute the definition of \( L(x) \) into the expectation:
\begin{equation*}
\mathbb{E}[L(X_t)] = \sum_{i=1}^{M} \int_{B_i} \left( \int_{B_i} \|x - x'\|_1 \hat{\pi}_{y_i}(dx') \right) \mu_t(dx).
\end{equation*}
Next, we interchange the order of the integrals, which is justified due to the non-negativity of the terms by the Fubini–Tonelli theorem:
\begin{equation*}
\mathbb{E}[L(X_t)] = \sum_{i=1}^{M} \int_{B_i} \int_{\mathds{X}}\mathds{1}_{B_i}(x) \|x - x'\|_1  \mu_t(dx) \hat{\pi}_{y_i}(dx')
\end{equation*}
where \( \mathds{1}_{B_i}(x) \) is the indicator function.

Now, we return to the full expression for the error bound in \eqref{eq:stationary_policy}. Then, \(\sum_{t=0}^{\infty} \beta^t \mathbb{E}[L(X_t)]\) becomes:
\begin{align}
\sum_{i=1}^{M} \int_{B_i} \left( \sum_{t=0}^{\infty} \beta^t \int_{\mathds{X}} \mathds{1}_{B_i}(x) \|x - x'\|_1 \mu_t^{\gamma_s}(dx) \right) \hat{\pi}_{y_i}(dx').
\label{eq:31}
\end{align}
For a fixed $x'$, define \( \tilde{c}(x) := \mathds{1}_{B_i}(x) \|x - x'\|_1 \):
\begin{align*}
\sum_{t=0}^{\infty} \beta^t \mathbb{E}^{\gamma_s}_{x_0} \left[ \tilde{c}(X_t) \right] &= \int_{\mathds{X}}\sum_{t=0}^{\infty} \beta^t \mathbb{E}^{\gamma_s}_{x_0} \left[ \tilde{c}(X_t) \mathds{1}_{\{ X_t \in dx \}} \right] = \int_{\mathds{X}}\tilde{c}(x)\sum_{t=0}^{\infty} \beta^t \mathbb{E}^{\gamma_s}_{x_0} \left[  \mathds{1}_{\{ X_t \in dx \}} \right] \nonumber \\
&= \int_{\mathds{X}}\tilde{c}(x)\sum_{t=0}^{\infty} \beta^t \mathds{P}^{\gamma_s}_{x_0} \left(X_t \in dx\ \right) = \int_{\mathds{X}}\tilde{c}(x)\nu^{\gamma_s}_{x_0}(dx) = \frac{1}{1 - \beta}\int_{\mathds{X}}\tilde{c}(x)\mu_{\beta}^{\gamma_s}(dx),
\end{align*}
where \(\mu_{\beta}^{\gamma_s}(A) := (1-{\beta})\nu^{\gamma_s}_{x_0}(A)\) for \( A \subseteq \mathds{X} \) and \(\nu^{\gamma_s}_{x_0} \) is the discounted occupation measure as we defined earlier.
We recognize this expression as a dot product between the cost function \( \tilde{c}(x) \) and the normalized occupation measure \( \mu_\beta^{\gamma_s}(dx) \), that is:
\begin{align*}
\langle \mu_\beta^{\gamma_s}, \tilde{c} \rangle = \int_{\mathds{X}} \tilde{c}(x) \mu_\beta^{\gamma_s}(dx),
\end{align*}
which leads to a linear program. Thus, the discounted sum can be expressed as:
\[
\sum_{t=0}^{\infty} \beta^t \mathbb{E}^{\gamma_s}_{x_0} \left[ \tilde{c}(X_t) \right] = \frac{1}{1 - \beta} \langle \mu_{\beta}^{\gamma_s}, \tilde{c} \rangle.
\]
The full expression then, by considering the distribution on the realizations for $x'$, for the discounted sum of the expected loss
function becomes:
\begin{align}\label{eq:expected_loss}
&\frac{1}{1 - \beta} \sum_{i=1}^{M} \int_{B_i} \int_{\mathds{X}} \|x - x'\|_1 \mathds{1}_{B_i}(x) \mu_{\beta}^{\gamma_s}(dx) \hat{\pi}_{y_i}(dx')\nonumber\\
&=\frac{1}{1 - \beta} \int_{\mathds{X}}  \sum_{i=1}^{M}  \mathds{1}_{B_i}(x) \int_{B_i}  \|x - x'\|_1 \hat{\pi}_{y_i}(dx')\mu_{\beta}^{\gamma_s}(dx) =\frac{1}{1 - \beta} \int_{\mathds{X}} L(x) \mu_{\beta}^{\gamma_s}(dx).
\end{align}

%Finally, we introduce the \( L_1 \) centroids, which are the medians of each quantization bin \( B_i \). In this case, the normalized weighting measure \( \hat{\pi}^*_{y_i}(dx') \) becomes a Dirac measure \( \delta_{y_i}(dx') \), concentrated at the centroid \( y_i \). This simplifies the expression significantly:
%\begin{equation}
%\frac{1}{1 - \beta} \sum_{i=1}^{M} \int_{B_i} \|x - y_i\|_1 \mu_{\beta}^{\gamma_s}(dx),
%\end{equation}

%where \( y_i \) is the \( L_1 \) centroid (median) of bin \( B_i \).

%In conclusion, the final refined upper bound for the quantization error becomes:
%\begin{equation}
%\sup_{\gamma_s \in \Gamma_s} \sum_{t=0}^{\infty} \beta^t \mathbb{E}^{\gamma_s}_{x_0} [L(X_t)] = \sup_{\gamma_s \in \Gamma_s} \frac{1}{1 - \beta} \sum_{i=1}^{M} \int_{B_i} \|x - y_i\|_1 \mu_{\beta}^{\gamma_s}(dx),
%\end{equation}
%where \( y_i \) is the \( L_1 \) centroid of each bin \( B_i \), and \( \mu_{\beta}^{\gamma_s}(dx) \) is the normalized discounted occupation measure. 
\end{proof}

\section{Proof of Theorem~\ref{theorem:uniform_quantizer_lyapunov}}
\label{appendix:uniform_quantizer_lyapunov}
\begin{proof}
First consider the case with $n=1$, that is $\mathds{X} \subseteq \mathbb{R}$
\textbf{Step 1:}
%Let $\mathds{X} \subseteq \mathbb{R}$ be the state space, which we p
Partition $\mathds{X}$ into $M+1$ quantization bins $\{ B_1, B_2, \ldots, B_{M}, B_{M+1} \}$. The first $M$ bins cover a compact subset $\mathcal{K} = [-\frac{N}{2},\frac{N}{2}] \subset \mathds{X}$, and the last bin $B_{M+1}$ is an overflow bin that captures the rest of the state space outside $\mathcal{K}$. We apply a uniform quantizer to the compact region $\mathcal{K}$, dividing it into $M$ bins of equal length. The quantization width of each bin is:
\begin{equation*}
\Delta = \frac{N}{M}.
\end{equation*}

\textbf{Step 2:}
For any state $x \in \mathcal{K}$, the quantization error $L(x)$ satisfies $L(x) \leq {\Delta}.$
We decompose the expected loss into two parts:
\begin{equation}
\mathbb{E}_{\mu_\beta^{\gamma_s}}[L(X)] = \int_{\mathcal{K}} L(x) \, \mu_\beta^{\gamma_s}(dx) + \int_{B_{M+1}} L(x) \, \mu_\beta^{\gamma_s}(dx),
\end{equation}
where $\mu_\beta$ is the normalized discounted occupation measure under a stationary policy $\gamma_s$. Since $L(x) \leq {\Delta}$ for $x \in \mathcal{K}$, we have:
\begin{equation}
\int_{\mathcal{K}} L(x)  \mu_\beta^{\gamma_s}(dx) \leq {\Delta} \cdot \mu_\beta^{\gamma_s}(\mathcal{K}) = {\Delta} \cdot \left( 1 - \mu_\beta^{\gamma_s}(B_{M+1}) \right).
\label{eq:compact_region_loss_learningg}
\end{equation}

\textbf{Step 3:}
In the overflow bin $B_{M+1}$, the state space may be unbounded, we take that the overflow bin is always mapped to state $x=0$: Thus, 
\begin{align}\label{overflow_control}
\int_{B_{M+1}} L(x) \, \mu_\beta^{\gamma_s}(dx) &= \int_{B_{M+1}} \int_{B_{M+1}}| x - x^{\prime} | \hat{\pi}_{M+1}(dx') \mu_\beta^{\gamma_s}(dx)  \nonumber\\ 
& \leq \int_{B_{M+1}}  \int_{B_{M+1}} \left( | x | + | x' | \right ) \hat{\pi}_{M+1}(dx^{\prime}) \mu_\beta^{\gamma_s}(dx) \nonumber\\
&= \mathbb{E}_{\mu_\beta^{\gamma_s}}[ | X | \mathds{1}_{B_{M+1}}(X)] ,
\end{align}
where $\hat{\pi}^*_{M+1}(dx')$ is the  normalized weighing measure over the bin $B_{M+1}$ and where the last step follows since we map the overflow bin directly to $0$, i.e. $x'=0$ with probability 1. Using H\"older's inequality to bound the expected loss over $B_{M+1}$:
\begin{equation}
\int_{B_{M+1}} | x | \, \mu_\beta^{\gamma_s}(dx) = \mathbb{E}_{\mu_\beta^{\gamma_s}}[ | X | \mathds{1}_{B_{M+1}}(X) ] \leq \left( \mathbb{E}_{\mu_\beta^{\gamma_s}}[ | X |^m ] \right)^{1/m} \mu_{\beta}^{\gamma_s}(B_{M+1})^{1 - 1/m}.
\label{eq:holder_application_learning}
\end{equation}

\textbf{Step 4:}
In the following, we bound the moment of the probability measure $\mu_{\beta}^{\gamma_s}$. 
{Define the process $\{M_t\}$ for $t \geq0$:
\[
M_t := \frac{V(x_t)}{(1 - \alpha)^t} - \sum_{k=0}^{t} \frac{b}{(1 - \alpha)^k},
\]
with $M_0 = V(x_0) - b = |x_0|^m-b.$ Observe the following with respect to the filtration $\{\mathcal{F}_t\}$, where $\mathcal{F}_t = \sigma(X_1,X_2,\dots,X_t)$ is the natural filtration:
\begin{align*}
 \mathbb{E}[M_{t+1} \mid \mathcal{F}_t] - M_t &= \mathbb{E}\left[\frac{V(X_{t+1})}{(1 - \alpha)^{t+1}} -\sum_{k=0}^{t+1} \frac{b}{(1 - \alpha)^k} \mid\mathcal{F}_t\right] - \frac{V(x_t)}{(1 - \alpha)^t} + \sum_{k=0}^{t} \frac{b}{(1 - \alpha)^k} \\
 & =  \frac{\mathbb{E}\left[V(X_{t+1})\mid\mathcal{F}_t\right]}{(1 - \alpha)^{t+1}} -\sum_{k=0}^{t+1} \frac{b}{(1 - \alpha)^k}  - \frac{V(x_t)}{(1 - \alpha)^t} + \sum_{k=0}^{t} \frac{b}{(1 - \alpha)^k}\\
 & \leq \frac{(1 - \alpha)V(x_t) }{(1 - \alpha)^{t+1}} - \frac{b}{(1 - \alpha)^{t+1}} -\frac{V(x_t)}{(1 - \alpha)^t} + \frac{b}{(1 - \alpha)^{t+1}} =0,
\end{align*}
where the inequality comes from the Lyapunov condition. Hence, we showed that the process $\{M_t\}$ is a supermartingale for all $t\geq0$. Then, observe that for every fixed stopping time $t$:
\begin{align*}
\mathbb{E}[M_t\mid\mathcal{F}_0] = \frac{\mathbb{E}\left[V(X_{t})\right]}{(1 - \alpha)^{t}} -\sum_{k=0}^{t} \frac{b}{(1 - \alpha)^k} \leq M_0 = V(x_0) - b.
\end{align*}
Rearranging the terms, we obtain:
\begin{align*}
\mathbb{E}[V(X_t)] &\leq (V(X_0) - b)(1 - \alpha)^t + (1 - \alpha)^t \sum_{k=0}^{t} \frac{b}{(1 - \alpha)^k}\\
&=  (V(X_0) - b)(1 - \alpha)^t + b \sum_{k=0}^{t} (1 - \alpha)^k = (V(X_0) - b)(1 - \alpha)^t + b \cdot \frac{1 - (1 - \alpha)^{t+1}}{\alpha}.
\end{align*}
Note that
\begin{align}
\mathbb{E}_{\mu_\beta^{\gamma_s}}[|X|^m]=\mathbb{E}_{\mu_\beta^{\gamma_s}}[V(X)] &= (1 - \beta) \sum_{t=0}^{\infty} \beta^t \mathbb{E}[V(X_t)] \nonumber \\
&\leq (1 - \beta) \sum_{t=0}^{\infty} \beta^t \left[ (V(x_0) - b)(1 - \alpha)^t + \frac{b}{\alpha}(1 - (1 - \alpha)^{t+1}) \right] \nonumber\\
&= (1 - \beta)(V(x_0) - b) \sum_{t=0}^{\infty} \left[\beta(1 - \alpha)\right]^t + \frac{b(1 - \beta)}{\alpha} \sum_{t=0}^{\infty} \beta^t \left[1 - (1 - \alpha)^{t+1}\right]\nonumber\\
& = \frac{(1 - \beta)(V(x_0) - b)}{1-\beta(1-\alpha)} + \frac{b(1-\beta)}{\alpha}\left[\frac{1}{1-\beta}-\frac{1-\alpha}{1-\beta(1-\alpha)} \right] \nonumber\\
&=\frac{V(x_0)(1-\beta)+b\beta}{1-\beta(1-\alpha)} =: C\label{eq:C},
\end{align}}

\textbf{Step 5:} Returning back to equation \eqref{eq:holder_application_learning}, we find:
\begin{align}
\int_{B_{M+1}} | x | \, \mu_\beta^{\gamma_s}(dx) = \mathbb{E}_{\mu_\beta^{\gamma_s}}[ | X | \mathds{1}_{B_{M+1}}(X) ] &\leq \left( \mathbb{E}_{\mu_\beta^{\gamma_s}}[ | X |^m ] \right )^{1/m} \mu_\beta^{\gamma_s}(B_{M+1})^{1 - 1/m}\nonumber\\
&\leq C^{1/m} \cdot \mu_\beta^{\gamma_s}(B_{M+1})^{1 - 1/m},
\label{eq:overflow_bin_loss}
\end{align}
where $C$ is defined in \eqref{eq:C}. Combining the bounds from \eqref{eq:compact_region_loss_learningg}, (\ref{overflow_control}) and \eqref{eq:overflow_bin_loss}:
\begin{equation}
\mathbb{E}_{\mu_\beta^{\gamma_s}}[ L(X) ] \leq {\Delta} \cdot \left( 1 - \mu_\beta^{\gamma_s}(B_{M+1}) \right) + C^{1/m} \cdot \mu_\beta^{\gamma_s}(B_{M+1})^{1 - 1/m} .
\label{eq:total_expected_loss_bound}
\end{equation}
%In particular, one can choose the weighing measure such that $\hat{\pi}^*_{M+1}(\cdot) = \mathds{1}_{\frac{N}{2}}(\cdot)$ which yields
%\begin{equation}
%\mathbb{E}_{\mu_\beta^{\gamma_s}}[ L(X) ] \leq \frac{\Delta}{2} \cdot \left( 1 - \mu_\beta^{\gamma_s}(B_{M+1}) \right) + C^{1/m} \cdot \mu_\beta^{\gamma_s}(B_{M+1})^{1 - 1/m} + \frac{N}{2} \cdot \mu_\beta^{\gamma_s}(B_{M+1}) .
%\label{eq:total_expected_loss_bound}
%\end{equation}

\textbf{Step 6:} In particular, if $N$ is chosen to be $N = 2(CM)^{\frac{1}{m}}$, then using Markov's inequality, we obtain:
\begin{equation}
\mu_\beta^{\gamma_s}(B_{M+1}) = \mathbb{P}( | X | \geq N/2 ) \leq \frac{ \mathbb{E}[ | X |^m ] }{ \left( (C M)^{1/m} \right)^m } \leq  \frac{ C }{ ( C M ) } = \frac{ 1 }{ M }.
\end{equation}
and 
\begin{equation*}
\Delta = \frac{ N }{ M  } = \frac{ 2 ( C M )^{1/m} }{ M } = \frac{ 2 C^{1/m} }{ M ^{ 1 - 1/m } }.
\end{equation*}
Using the bounds we obtained for the particular choice of $N$ in \eqref{eq:total_expected_loss_bound}, we get: 
\begin{align}
\mathbb{E}_{\mu_\beta^{\gamma_s}}[ L(X) ] \leq \frac{2 C^{1/m} }{ M ^{ 1 - 1/m } }  + C^{1/m} \cdot \left( \frac{1}{M} \right )^{1 - 1/m} 
\leq \frac{ 3 C^{1/m} }{ M^{  1 - 1/m  }}.
\end{align}
With $m > 1$, $1 - 1/m > 0$, and thus as $M \to \infty $, the total expected loss  $\mathbb{E}_{\mu_\beta^{\gamma_s}}[ L(X) ] \to 0$. By combining this bound with Theorem \ref{theorem:optimized_error_bound}, we obtain:
\begin{equation*}
\left| \hat{J}_{\beta}(x_0) - J^*_{\beta}(x_0) \right| \leq \left( \alpha_c + \frac{ \beta \alpha_T \| c \|_{\infty} }{ 1 - \beta } \right) \frac{ 3 C^{1/m} }{ (M^{ 1 - 1/m }) (1-\beta) },
\end{equation*}
where $M$ is the number of quantization bins, $m>1$ is the moment in the Lyapunov function, and $C$ is (defined in \eqref{eq:C}) the uniform bound on the $m$-th moment of $X_t$ with respect to measure $\mu_\beta^{\gamma_s}$.
%\end{proof}
%\section{Proof of Theorem~\ref{theorem:uniform_quantizer_lyapunov_nd}}
%\label{appendix:uniform_quantizer_lyapunov_n}
%\begin{proof}

This completes the case with $n=1$. We now consider the case where $\mathds{X}\subseteq\mathbb{R}^{n}$ with $n \geq 1$.

Now, fix a side length parameter $N>0$. Choose an integer $k\ge1$ and set 
\[
    M := k^{n}\qquad\bigl(\text{so that }k=M^{1/n}\bigr).
\]
Define the centered $n$–dimensional hyper-cube
\[
    \mathcal{K}:=
    \Bigl[-\frac{N}{2},\frac{N}{2}\Bigr]^{n}
    \subset\mathds{X}.
\]
Partitioning each dimension of $\mathcal{K}$ with k bins uniformly yields a quantization bin with width:
\[
    \Delta = \frac{N}{k} = \frac{N}{M^{1/n}} ,
\]
thus producing $M$ congruent hyper-cubic cells of volume
$\Delta^{\,n}=N^{n}/M$.  Index these cells by
\[
    B_{i_1,\dots,i_n}
    \;=\;
    \prod_{j=1}^{n}
          \Bigl[
            -\tfrac{N}{2}+(i_j-1)\Delta,\;
            -\tfrac{N}{2}+i_j\Delta
          \Bigr),
    \quad
    i_j\in\{1,\dots,k\},
\]
and enumerate them as $\{B_{1},\dots,B_{M}\}$. Set
\[
    B_{M+1}\;:=\mathds{X}\setminus\mathcal{K},
\]
which is the set of states outside the compact (granular) grid. We decompose the expected loss into two parts as in the proof of Theorem \ref{theorem:uniform_quantizer_lyapunov}:
\begin{equation}\label{eq:loss_decomp_nd}
    \mathbb{E}_{\mu_{\beta}^{\gamma_{s}}}[L(X)]
    = \int_{\mathcal{K}} L(x)\mu_{\beta}^{\gamma_{s}}(dx)
    +\int_{B_{M+1}} L(x)\mu_{\beta}^{\gamma_{s}}(dx),
\end{equation}
where
\(
  \mu_{\beta}^{\gamma_{s}}
\)
is the normalized discounted occupation measure under the stationary
policy $\gamma_{s}$.

Inside $\mathcal{K}$ the error is uniformly bounded, hence
\begin{equation}\label{eq:interior_loss_nd}
    \int_{\mathcal{K}}L(x)\mu_{\beta}^{\gamma_{s}}(dx)\leq
    {n\,\Delta\,\mu_{\beta}^{\gamma_{s}}\!(\mathcal{K})}
    = {n\,\Delta\bigl[1-\mu_{\beta}^{\gamma_{s}}(B_{M+1})\bigr]}.
\end{equation}

Inside the overflow bin \(B_{M+1}\subseteq\mathds{X}\), we take again that the overflow bin is directly mapped to state $x=0$ and thus, 
\begin{align*}
 &  \int_{B_{M+1}} L(x)\mu_{\beta}^{\gamma_{s}}(dx) = \int_{B_{M+1}} \int_{B_{M+1}} \|x-x'\|_{1} \hat\pi^{*}_{M+1}(dx') \mu_{\beta}^{\gamma_{s}}(dx) \\
  &\le\int_{B_{M+1}}\int_{B_{M+1}}\bigl(\|x\|_{1}+\|x'\|_{1}\bigr)\hat\pi^{*}_{M+1}(dx') \mu_{\beta}^{\gamma_{s}}(dx)  =\mathbb{E}_{\mu_{\beta}^{\gamma_{s}}}\bigl[\|X\|_{1}\mathbf{1}_{B_{M+1}}(X)\bigr]
  \end{align*}
  where the last step follows since we assume that $x'=0$ always for the overflow bin.
We apply Hölder’s inequality with exponent \(m>1\):
\begin{equation}\label{eq:holder_nd}
    \mathbb{E}_{\mu_{\beta}^{\gamma_{s}}}\bigl[\|X\|_{1}\mathbf{1}_{B_{M+1}}(X)\bigr]\le\Bigl(\mathbb{E}_{\mu_{\beta}^{\gamma_{s}}}\bigl[\|X\|_{1}^{m}\bigr]\Bigr)^{1/m}
    \mu_{\beta}^{\gamma_{s}}(B_{M+1})^{1-1/m}.
\end{equation}
Repeating the arguments in {\bf Step 4},
\[
\mathbb{E}_{\mu_\beta^{\gamma_s}}[\|X\|_1^m]\leq\frac{V(x_0)(1-\beta)+b\beta}{1-\beta(1-\alpha)} =: C
\]
Combining~\eqref{eq:interior_loss_nd} with the overflow estimate
obtained in~\eqref{eq:holder_nd}, we get:
\begin{equation}
    \mathbb{E}_{\mu_{\beta}^{\gamma_{s}}}[L(X)]
    \le{n\Delta
      \bigl[1-\mu_{\beta}^{\gamma_{s}}(B_{M+1})\bigr] }+C^{1/m}\mu_{\beta}^{\gamma_{s}}(B_{M+1})^{1-1/m}.
            \label{eq:total_loss_split_nd}
\end{equation}
We select
\(
    N = 2\bigl(C\,k\bigr)^{1/m}.
\)
By Markov's Inequality:
\begin{align}
    \mu_{\beta}^{\gamma_{s}}(B_{M+1})&={\mu_{\beta}^{\gamma_{s}}(\{x = {x_1,\cdots,x_n}: \min_{i=1,\cdots,n} |x_i| \geq N/2\})   \leq }  \Pr\bigl[\|X\|_{1}\ge N/2\bigr] \nonumber \\
&   \le \frac{\mathbb{E}_{\mu_{\beta}^{\gamma_{s}}}[\|X\|_{1}^{m}]}
    {(\tfrac{N}{2})^{m}}=\frac{C}{C\,k}=\frac{1}{k}. \nonumber 
\end{align}
With \(k=M^{1/n}\) interior cells per axis, the hyper-cube side length is
\[ \Delta =\frac{N}{k}=2C^{\frac{1}{m}}k^{(1-\frac{1}{m})}.
\]
%Additionally, we can choose the weighing measure such that $\hat\pi^{*}_{M+1}(\cdot)=\mathds{1}_{[N/2,\dots,N/2]}(\cdot)$ such that $\mathbb{E}_{\hat\pi^{*}_{M+1}}\bigl[\|X'\|_{1}\bigr] = \frac{n.N}{2} $.
Inserting these  bounds into~\eqref{eq:total_loss_split_nd}:
\begin{align*}
& \mathbb{E}_{\mu_{\beta}^{\gamma_{s}}}[L(X)]
      \le 2nC^{1/m}\left(\frac{1}{k}\right)^{1-\frac{1}{m}}+C^{1/m}\left(\frac{1}{k}\right)^{1-\frac{1}{m}} \le (2n+1)C^{1/m}k^{-(1-\frac{1}{m})}\end{align*}
 As $M \to \infty $, and so as $k\to\infty$, we have $\mathbb{E}_{\mu_\beta^{\gamma_s}}[ L(X) ] \to 0$. Thus,
\begin{equation}
\left| \hat{J}_{\beta}(x_0) - J^*_{\beta}(x_0) \right| \leq \left( \alpha_c + \frac{ \beta \alpha_T \| c \|_{\infty} }{ 1 - \beta } \right) \frac{ (2n+1) C^{1/m} }{ (M^{ \frac{1}{n}(1-\frac{1}{m}) }) (1-\beta) },
\end{equation}
where $n$ is the dimension of the state space, $M$ is the number of quantization bins, $m$ is the moment in the Lyapunov function,  and $C$ is (defined in \eqref{eq:C}) the uniform bound on the $m$-th moment of $X_t$ with respect to measure $\mu_\beta^{\gamma_s}$.
\end{proof}

\section{ Proof of Theorem \ref{theorem:error_bound_average_alt}} \label{appendix:error_bound_average_alt}

\begin{proof}
Using Assumption \ref{assumption:minorization}, the following Average Cost Optimality Equation (ACOE) is satisfied for the original model:
\begin{equation} \label{equation:acoe_minorized}
    h(x) = \inf_{u \in \mathds{U}} \left\{ c(x,u) + \int_{\mathds{X}} h(x_1) \mathcal{T}(dx_1 | x,u) \right\} - \int_{\mathds{X}} h(x) \mu(dx).
\end{equation}
Similarly, for the discretized model, we have the corresponding ACOE:
\begin{equation}
\hat{h}(y) = \inf_{u \in \mathds{U}} \left\{ C^*(y,u) + \sum_{y_1} \hat{h}(y_1) P^*(y_1 | y,u) \right\} - \int_{\mathds{X}} \hat{h}(q(x)) \mu(dx).
\end{equation}
According to Verification Theorem for average cost criterion (see \cite{hernandez1999further} and \cite{arapostathis1993discrete}), the long-term average cost \( J^*(x) \) for the original model and \( \hat{J}(x) \) for the discretized model are given by:
\[
J^*(x) = j^* = \int_{\mathds{X}} h(x) \mu(dx),
\]
\[
\hat{J}(x) = \hat{j} = \int_{\mathds{X}} \hat{h}(x) \mu(dx),
\]
for all \( x \in \mathds{X} \), where $h(\cdot)$ and $\hat{h}(\cdot)$ are the fixed-point solutions of the equations \eqref{equation:acoe_minorized} and \eqref{equation:acoe_minorized_approximate}, respectively.

For the rest of the proof, we simply use the same notation for $C^*$, and $\hat{h}$ by extending them as constant over the quantization bins, that is we override the notation and use $ \hat{h}(x)$ for $\hat{h}(q(x)) $ and $C^*(x,u)$ for $C^*(q(x),u)$. 

By defining a version of the finite model kernel $P^*$ that is defined over $\mathds{X}$ such that
\begin{align*}
\hat{\mathcal{T}}(\cdot|x,u):= \int_{B_i} \mathcal{T}(\cdot|x',u)\hat{\pi}_i(dx') \text{ for } x \in B_i,
\end{align*}
one can show that 
\begin{align}\label{equation:acoe_minorized_approximate}
\hat{h}(x) = \inf_{u \in \mathds{U}} \left\{ C^*(x,u) + \int \hat{h}(x_1) \hat{\mathcal{T}}(dx_1 | x,u) \right\} - \int_{\mathds{X}} \hat{h}(x) \mu(dx).
\end{align}

We denote by $\mathcal{T}^-(\cdot|x,u):=\mathcal{T}(\cdot|x,u)-\mu(\cdot)$ and $\hat{\mathcal{T}}^-(\cdot|x,u):=\hat{\mathcal{T}}(\cdot|x,u)-\mu(\cdot)$. Note that 
\begin{align*}
\|\mathcal{T}^-(\cdot|x,u) - \hat{\mathcal{T}}^-(\cdot|x,u)\|_{TV} = \|\mathcal{T}(\cdot|x,u) - \hat{\mathcal{T}}(\cdot|x,u)\|_{TV}. 
\end{align*}
We further denote by $V(x):=|h(x)-\hat{h}(x)|$. In what follows, the term $\sup_u \int V(x_1)\mathcal{T}(dx_1|x,u)$ will be of interest. We will denote the control function that achieves the supremum by $\gamma_s$, whose existence is guaranteed under Assumption \ref{assumption:lipschitz_cost_and_transitions}. Comparing the ACOEs for corresponding models, we can write that
\begin{align*}
V(x) &\leq \sup_{u\in \mathds{U}}\bigg| (c(x,u) - C^*(x,u)) +  \left(\int h(x_1)\mathcal{T}^-(dx_1|x,u) - \int \hat{h}(x_1)\hat{\mathcal{T}}^-(dx_1|x,u) \right)\bigg|\\
& \leq \alpha_c L(x) + \sup_{u\in\mathds{U}}\bigg| \int h(x_1)\mathcal{T}^-(dx_1|x,u) - \int \hat{h}(x_1)\mathcal{T}^-(dx_1|x,u) \bigg|\\
&\qquad\qquad + \sup_{u\in\mathds{U}}\bigg| \int \hat{h}(x_1)\mathcal{T}^-(dx_1|x,u) - \int \hat{h}(x_1)\hat{\mathcal{T}}^-(dx_1|x,u) \bigg|\\
&\leq \alpha_c L(x)  + \sup_{u\in\mathds{U}}\int V(x_1)\mathcal{T}^-(dx_1|x,u) + \|\hat{h}\|_\infty \alpha_T L(x)\\
&= (\alpha_c + \|h\|_\infty \alpha_T)L(x) + \int V(x_1)\mathcal{T}(dx_1|x,\gamma_s(x)) -\int V(x)\mu(dx).
\end{align*}
By repeating the same step, one can write that for any $T<\infty$:
\begin{align*}
V(x)\leq (\alpha_c + \|h\|_\infty \alpha_T) \sum_{t=0}^{T-1}E^{\gamma_s}\left[L(X_t)\right] - T \int V(x)\mu(dx).
\end{align*}
We can then write that 
\begin{align*}
\int V(x)\mu(dx) + \frac{V(x)}{T} \leq  (\alpha_c + \|h\|_\infty \alpha_T) \frac{1}{T}\sum_{t=0}^{T-1}E^{\gamma_s}\left[L(X_t)\right] .
\end{align*}
At the end of the proof, we will show that $h(x),\hat{h}(x)$ and thus $V(x)$ are uniformly bounded. Assuming this is true and sending $T\to\infty$, we get
\begin{align*}
\int V(x)\mu(dx) + \frac{V(x)}{T}& \leq  (\alpha_c + \|h\|_\infty \alpha_T) \limsup_{T\to\infty}\frac{1}{T}\sum_{t=0}^{T-1}E^{\gamma_s}\left[L(X_t)\right] \\
&= (\alpha_c + \|h\|_\infty \alpha_T) \int L(x) \pi_{\gamma_s}(dx),
\end{align*}
where $\pi_{\gamma_s}$ is the invariant measure induced by policy $\gamma_s$ which exists under Assumption \ref{assumption:minorization}.

To show the boundedness, we let $\hat{T}$ denote the ACOE operator for the finite model. Starting from $h_0(x)= 0$, let $h_k$ denote the function we obtain after applying the operator $\hat{T}$, $k$ consecutive times. We then have that 
\begin{align*}
h_k(x)\leq \|c\|_\infty + \sup_u\left\{\int h_{k-1}(x_1)\hat{\mathcal{T} }^-(dx_1|x,u) \right\}\leq \|c\|_\infty + \|h_{k-1}\|_\infty (1-\mu(\mathds{X})).
\end{align*}
Letting $\alpha:=(1-\mu(\mathds{X}))$, and repeating this step, we write $\|h_k\|_\infty \leq \|c\|_\infty \sum_{t=0}^{k-1} \alpha^t$. Finally, using the fact that the operator $\hat{T}$ is a contraction under the supremum norm under Assumption \ref{assumption:minorization} with the fixed point $\hat{h}$, we conclude $\|\hat{h}\|_\infty \leq \frac{\|c\|_\infty}{1-\alpha} = \frac{\|c\|_\infty}{\mu(\mathds{X})}$. Identical steps can be used to show that $\|h\|_\infty\leq \frac{\|c\|_\infty}{\mu(\mathds{X})}$. Combining what we have so far, we write 
\begin{align*}
|j^*-\hat{j}| = \left|\int h(x)\mu(dx) - \int \hat{h}(x)\mu(dx)\right| \leq \int V(x)\mu(dx) \leq \left(\alpha_c + \frac{\|c\|_\infty \alpha_T}{\mu(\mathds{X})}\right) \int L(x) \pi_{\gamma_s}(dx).
\end{align*}

\end{proof}

\section{Proof of  Theorem~\ref{theorem:uniform_quantizer_lyapunov_average}}
\label{appendix:uniform_quantizer_lyapunov_average}

\begin{proof}
 %   First we assume that the state space $\mathds{X} \subseteq \mathbb{R} $ and let $b \geq 0$, $V:\mathds{X} \to[0,\infty)$, $f:\mathds{X} \to [\epsilon,\infty)$ for some $\epsilon>0$. Assume the state process $\{X_t\}$ satisfies the following condition:
%\begin{equation}
%\sup_{x \in \mathds{X},u \in \mathds{U}} \mathbb{E}[V(X_{t+1})|X_t=x,U_t=u] \leq V(x) - f(x) + b,\label{eq:lyapunov_eq_avg}
%\end{equation}
%where $f(x)=|x|^m$, $m>1$.
Choose a collection of quantization bin $\{B_i\}_i^{M+1}$ such that the first $M$ bins quantize the compact set $[- ( b k )^{1/m}, ( b k )^{1/m}]^n$ uniformly where $M=k^n$ and the last bin $B_{M+1}$ is the overflow bin that captures the rest of the state space. Then, the total expected distortion becomes:
$$
\mathbb{E}_{\pi_{\gamma_s}}[L(X)] = \int_{\mathcal{K}} L(x) \, \pi_{\gamma_s}(dx) + \int_{B_{M+1}} L(x) \, \pi_{\gamma_s}(dx)
$$

%In the overflow bin $B_{M+1}$, the state space may be unbounded, and the quantization error $L(x) = \| x - y_{M+1} \|_1$ can be large. We need to bound:
%\begin{equation}
%\int_{B_{M+1}} L(x) \, \pi_{\gamma_s}(dx) = \int_{B_{M+1}} \| x - y_{M+1} \|_1 \, \pi_{\gamma_s}(dx),
%\end{equation}
%where $y_M$ is the representative point for bin $B_{M+1}$. 

As earlier, we choose $y_{M+1} = 0$ (or any fixed point), so that $L(x) = \| x \|_1$ for $x \in B_{M+1}$. By our assumption, there exists a Lyapunov function $V$ that satisfies inequality \eqref{eq:lyapunov_eq_avg}. Thus, by \cite[Theorem 4.2.5]{yuksel2020control} (which builds critically on the Comparison Theorem \cite[Theorem 14.2.2]{meyn1993markov}), under any invariant probability measure \( \pi \): $\int_\mathds{X} \|x\|_1^m \pi(dx)\leq b$

%Then, by Theorem~\ref{theorem:comparison_theorem}, for any initial state $x_0=x$, we have:
%\begin{equation*}
%\mathbb{E}_x\left[\sum_{t=0}^{T-1} f(x_t)\right] \leq V(x) + \mathbb{E}\left[\sum_{t=0}^{T-1} b\right] = V(x) + bT,
%\end{equation*}
%where $T$ is a deterministic stopping time. By taking the limit for both sides, we get:
%\begin{equation*}
%\limsup_{T \to \infty}\frac{1}{T}\mathbb{E}_x\left[\sum_{t=0}^{T-1} f(x_t)\right] \leq \limsup_{T \to \infty}\frac{1}{T} (V(x) + bT)=b ,
%\end{equation*}
%
%Now, consider any invariant probability measure \( \pi \). Recall that we assigned $f =|x|^m, m>1$. Fix \( N > 0 \) and define \( f_N(x) = \min(|x|^m, N) \). Then, we can write (see \cite[p. 75]{yuksel2024lecture}):
%\begin{align*}
%\int_\mathds{X} f_N(x) \pi(dx) &= \limsup_{T \to \infty} \frac{1}{T} \sum_{t=0}^{T-1} \mathbb{E}_{x_0\sim\pi}\left[E_{x_0}[f_N(X_t)]\right] \\
%&\leq \mathbb{E}_\pi \left[ \limsup_{T \to \infty} \frac{1}{T} \sum_{t=0}^{T-1} E_{x_0}[f_N(X_t)]\right] \leq b,
%\end{align*}
%where the first inequality comes from applying Fatou's Lemma since \( f_N(x) \) is bounded by \( N \). By Taking the limit $N \to \infty$ and noting that $ f_N(x) \to |x|^m $ pointwise and monotonically, by Monotone Convergence Theorem, we obtain $\int_\mathds{X} |x|^m \pi(dx)\leq b$.

%Finally, since this holds for any invariant probability measure $\pi$ and \( |x|^m \) is non-negative, we conclude that $    \mathbb{E}[ | X_t |^m ] \leq b,  \forall t \geq 0.$

By following the same procedure as in the proof of Theorem~\ref{theorem:uniform_quantizer_lyapunov}, we obtain the following bound:
\begin{align}
\mathbb{E}_{\pi_{\gamma_s}}[ L(X) ] &\leq \Delta \cdot \left( 1 - \pi_{\gamma_s}(B_{M+1}) \right) + b^{1/m} \cdot \pi_{\gamma_s}(B_{M+1})^{1 - 1/m} \\
& \leq\Delta \ + \frac{b^{1/m}}{k^{  1 - 1/m }} \leq \frac{ 2n b^{1/m} }{ k^{ 1 - 1/m } } + \frac{b^{1/m}}{k^{ 1 - 1/m}} = \frac{(2n+1) b^{1/m}}{k^{  1 - 1/m  }}.
\label{eq:total_expected_loss_bound_final_avg}
\end{align}
By combining this bound with Theorem \ref{theorem:error_bound_average_alt} under the given assumptions, we obtain:
\begin{equation}
\left| \hat{J}_{avg}(x_0) - J^*_{avg}(x_0) \right| \leq  \left(\alpha_c+ \frac{\alpha_T \|c\|_\infty}{\mu(\mathds{X})} \right) \frac{ (2n+1) b^{1/m} }{ (M^{1/n( 1 - 1/m )}) },
\end{equation}
where $M$ is the number of quantization bins, $m>1$ is the moment in Lyapunov function and $b$ is the uniform bound on the $m$-th moment of $X_t$. Thus, the average cost leads to sharper bounds via the Foster-Lyapunov analysis.
\end{proof}

\section{Proof of Theorem \ref{theorem:lyapunov_learning}}\label{proof:lyapunov_learning}

\begin{proof}

\textbf{Step1: } We partition the state space similar to the quantization in Theorem \ref{theorem:uniform_quantizer_lyapunov} by choosing the compact set $\mathcal{K}=[-(Ck)^{1/m},(Ck)^{1/m}]$ with $k^n=M$, ($k=M^{1/n}$). Then, for any state $x \in \mathcal{K}$, the quantization error $L(x)$ satisfies:
\begin{equation}
L(x) \leq {\Delta} = \frac{ 2n( C k )^{1/m} }{ k }=\frac{ 2nC^{1/m} }{ k ^{ 1 - 1/m } } = \frac{ 2nC^{1/m} }{ M ^{1/n (1 - 1/m) } }. 
\end{equation}

We decompose the total expected loss into two parts:
\begin{equation}
\mathbb{E}_{\mu_\beta^{\gamma_s}}[L(X)] = \int_{\mathcal{K}} L(x) \, \mu_\beta^{\gamma_s}(dx) + \int_{B_{M+1}} L(x) \, \mu_\beta^{\gamma_s}(dx),
\label{eq:expected_loss_decomposition_learning}
\end{equation}
where $\mu_\beta^{\gamma_s}$ is the normalized discounted occupation measure under a stationary policy $\gamma_s$.

\textbf{Step 2:} Since $L(x) \leq {\Delta}$ for $x \in \mathcal{K}$, we have:
\begin{equation}
\int_{\mathcal{K}} L(x) \, \mu_\beta^{\gamma_s}(dx) \leq {\Delta} \cdot \mu_\beta(\mathcal{K}) = \frac{ 2nC^{1/m} }{ M ^{1/n (1 - 1/m) } } 
\cdot \left( 1 - \mu_\beta^{\gamma_s}(B_{M+1}) \right).
\label{eq:compact_region_loss_learning}
\end{equation}

\textbf{Step 3 (Two parallel arguments):}  In the overflow bin $B_{M+1}$, the state space may be unbounded. Unlike the analysis in Section \ref{section:refined_discounted}, we cannot assign arbitrary weighting measure $\hat{\pi}_{M+1}$ for the overflow bin. Thus, the bound becomes:
\begin{align}
\int_{B_{M+1}} L(x) \, \mu_\beta^{\gamma_s}(dx) &= \int_{B_{M+1}} \int_{B_{M+1}}\| x - x^{\prime} \|_1 \hat{\pi}_{M+1}(dx') \mu_\beta^{\gamma_s}(dx)  \\ 
& \leq\int_{B_{M+1}}  \int_{B_{M+1}} \left( \| x \|_1 + \| x' \|_1 \right ) \hat{\pi}_{M+1}(dx^{\prime}) \mu_\beta^{\gamma_s}(dx) \\
&= \mathbb{E}_{\mu_\beta^{\gamma_s}}[ \| X \|_1 \mathds{1}_{B_{M+1}}(X)] + \mathbb{E}_{\hat{\pi}_{M+1}}[ \| X' \|_1] \cdot \mu_\beta^{\gamma_s}(B_{M+1}) ,
\end{align}
where $\hat{\pi}^*_{M+1}(dx')$ is the normalized invariant measure of the state process over the bin $B_{M+1}$ under the exploration policy. The second term is handled in two ways:

\paragraph{3.a}
Observe that if $x_0 \sim \pi_{\gamma^*}$ then $\mu_\beta^{\gamma_s} = \pi_{\gamma^*}$ and therefore the terms in the summation above will be identical, that is,
\[\mathbb{E}_{\hat{\pi}^*_{M+1}}[ \| X' \|_1] \cdot \mu_\beta^{\gamma_s}(B_{M+1})  =\mu_\beta^{\gamma_s}(B_{M+1}) \frac{\mathbb{E}_{\pi_{\gamma^*}}[ \| X \|_1\mathds{1}_{B_{M+1}}(X)]}{\pi_{\gamma^*}(B_{M+1})} = \mathbb{E}_{\mu_\beta^{\gamma_s}}[ \| X \|_1 \mathds{1}_{B_{M+1}}(X)] .\]

%In the proof of Theorem \ref{theorem:uniform_quantizer_lyapunov}, we showed that for a fixed initial state $x_0$: 
%\begin{equation*}
%  \mathbb{E}_{\mu_\beta^{\gamma_s}}[ | X | \mathds{1}_{B_{M+1}}(X)]  \leq  C^{1/m} \cdot \mu_\beta^{\gamma_s}(B_{M+1})^{1 - 1/m},
%\end{equation*}
%where 
%\[
%C:= \frac{|x_0|^m(1-\beta)+b\beta}{1-\beta(1-\alpha)}
%\]

Similar to the proof of Theorem \ref{theorem:uniform_quantizer_lyapunov}, by the law of the iterated expectations, we can show that when the initial state is not fixed but given by $x_0\sim\pi_{\gamma^*}$: 
\begin{equation*}
  \mathbb{E}_{\mu_\beta^{\gamma_s}}[ \| X \| \mathds{1}_{B_{M+1}}(X)] = \mathbb{E}_{\mu_\beta^{\gamma_s}}[ \mathbb{E}[ \| X \|_1 \mathds{1}_{B_{M+1}}(X)|X_0=x]]  \leq  \hat{C}^{1/m} \cdot \mu_\beta^{\gamma_s}(B_{M+1})^{1 - 1/m},
\end{equation*}
where 
\[
\hat{C}:= \frac{\mathbb{E}_{\mu_\beta^{\gamma_s}}[\|x_0\|_1]^m(1-\beta)+b\beta}{1-\beta(1-\alpha)}
\]

By using Markov's inequality, we obtain:
\begin{align}\label{eq:overflow_probability_bound}
    \mu_{\beta}^{\gamma_{s}}(B_{M+1})&={\mu_{\beta}^{\gamma_{s}}(\{x = {x_1,\cdots,x_n}: \min_{i=1,\cdots,n} |x_i| \geq (Ck)^{1/m}\})   \leq }  \Pr\bigl[\|X\|_{1}\ge (Ck)^{1/m}\bigr] \leq \frac{1}{k}  
\end{align}
which gives:
\begin{align}
\int_{B_{M+1}} L(x) \, \mu_\beta^{\gamma_s}(dx) 
&\le \frac{2\hat{C}^{1/m}}{k^{1-1/m}} = \frac{2\hat{C}^{1/m}}{M^{1/n(1-1/m)}} 
\end{align}

\paragraph{3.b}  We know that $\hat{\pi}_{M+1}$ is the normalized measure of the invariant distribution $\pi_{\gamma^*}$ under the exploration policy. Then, we have:
\begin{align*}
    \mathbb{E}_{\hat{\pi}_{M+1}}[ \| X \|_1] = \frac{\mathbb{E}_{\pi_{\gamma^*}}[ \| X \|_1\mathds{1}_{B_{M+1}}(X)]}{\pi_{\gamma^*}(B_{M+1})} \leq \frac{\mathbb{E}_{\pi_{\gamma^*}}[\|X\|^m]^{1/m} \cdot \pi_{\gamma^*}(B_{M+1})^{1-1/m}}{\pi_{\gamma^*}(B_{M+1})} = \frac{\mathbb{E}_{\pi_{\gamma^*}}[\|X\|^m]^{1/m}}{\pi_{\gamma^*}(B_{M+1})^{1/m}}
\end{align*} By \cite[Theorem 14.2.2]{meyn1993markov}, under the Lyapunov condition in \eqref{eq:lyapunov_eq_learning} we obtain:
\begin{align}
    \mathbb{E}_{\pi_{\gamma^*}}[\|X\|_1^m] \leq \frac{b}{\alpha}.
\end{align} 
For the overflow set  
\(B_{M+1}\) we have by Markov’s inequality:
\[
  \pi_{\gamma^{*}}(B_{M+1})\le
     \frac{\mathbb E_{\pi_{\gamma^{*}}}[\|X\|_1^{m}]}{(Ck)^{m/m}}
     \le\frac{b/\alpha}{Ck}.
\]
Putting the pieces together:
\begin{align*}
  \mathbb{E}_{\hat{\pi}_{M+1}}[\|X'\|_1]
  \mu_{\beta}^{\gamma_{s}}(B_{M+1}) \le \frac{\mathbb{E}_{\pi_{\gamma^*}}[\|X\|_1^m]^{1/m}}{\pi_{\gamma^*}(B_{M+1})^{1/m}} \cdot \frac{1}{k} =\Bigl(\tfrac{b}{\alpha}\Bigr)^{1/m}
     \Bigl(\tfrac{C}{b/\alpha}\Bigr)^{1/m}
     k^{-(1-1/m)} =\frac{C^{1/m}}{M^{1/n(1-1/m)}}
\end{align*}
and
\begin{align*}
\int_{B_{M+1}} L(x) \, \mu_\beta^{\gamma_s}(dx) 
&\le \frac{2C^{1/m}}{M^{1/n(1-1/m)}},
\end{align*}
which goes to zero as M goes to infinity. Thus the expected quantization error for the overflow bin vanishes as the number of quantization bin, M, increases. 

Either method gives the
same final error bound stated in Theorem~\ref{theorem:lyapunov_learning}. While the first method assumes that the initial state distribution is the same as $\pi_{\gamma^*}$, the second method assumes that the initial state can be any deterministic state in $\mathbb{X}$. 

\textbf{Step 4:} Now, combining the bounds we obtained so far, we get: 
\begin{align}
\mathbb{E}_{\mu_\beta^{\gamma_s}}[ L(X) ] &\leq \frac{  2C^{1/m} }{ M ^{1/n(1 - 1/m) } } \cdot \left( 1 - \mu_\beta^{\gamma_s}(B_{M+1}) \right)  + \frac{2C^{1/m}}{M^{1/n(1-1/m)}}  \nonumber\\
&\leq \frac{  2C^{1/m} }{ M^{1/n( 1 - 1/m) } } + \frac{2C^{1/m}}{M^{1/n( 1 - 1/m)}} = \frac{4 C^{1/m}}{M^{1/n(  1 - 1/m)  }}.
\label{eq:total_expected_loss_bound_final}
\end{align}
Obviously, $m > 1$ implies  $1 - 1/m > 0$.

\textbf{Step 5:} Let By combining this bound with Theorem \ref{theorem:optimized_error_bound} under the necessary assumptions, we obtain:
\begin{equation}
\left| \hat{J}_{\beta}(x_0) - J^*_{\beta}(x_0) \right| \leq \left( \alpha_c + \frac{ \beta \alpha_T \| c \|_{\infty} }{ 1 - \beta } \right) \left(\frac{ 4 C^{1/m} }{ (M^{ 1/n(1 - 1/m) }) (1-\beta) }\right),
\end{equation}
where $M$ is the number of quantization bins, $m>1$ is the moment in Lyapunov function, $C$ is the uniform bound on the $m$-th moment of $X_t$ with respect to measure $\mu_\beta$. 
\end{proof}

\bibliographystyle{plain} % 
\bibliography{SerdarBibliography,OsmanBibliography,AliBibliography} 

\end{document}